\documentclass[12pt,twoside]{amsart}
\usepackage{amssymb,amsmath,amsthm, amscd, enumerate, mathrsfs, calligra}
\usepackage[all]{xy}
\usepackage{comment}

\title[Ample canonical heights for endomorphisms]
{Ample canonical heights for endomorphisms on projective varieties}
\author{Takahiro Shibata} 
\date{}
\address{Department of Mathematics, Graduate School of Science, 
Kyoto University, Kyoto 606-8502, Japan}
\email{tshibata@math.kyoto-u.ac.jp}

\DeclareMathOperator{\Supp}{Supp}

\DeclareMathOperator{\id}{id}
\DeclareMathOperator{\Hom}{Hom}

\DeclareMathOperator{\tor}{tor}

\DeclareMathOperator{\sHom}{\mathscr{H}\kern -.3pt \mathit{om}}

\DeclareMathOperator{\End}{End}

\DeclareMathOperator{\Aut}{Aut}
\DeclareMathOperator{\tors}{tor}
\DeclareMathOperator{\Preper}{Preper}
\DeclareMathOperator{\Per}{Per}

\newtheorem{thm}{Theorem}[section]
\newtheorem{lem}[thm]{Lemma}
\newtheorem{cor}[thm]{Corollary}
\newtheorem{prop}[thm]{Proposition}
\newtheorem{conj}[thm]{Conjecture}
\newtheorem{ques}[thm]{Question}

\theoremstyle{definition}
\newtheorem{defn}[thm]{Definition}
\newtheorem{rem}[thm]{Remark}
\newtheorem{ex}[thm]{Example}
\newtheorem*{ack}{Acknowledgments} 

\newtheorem*{notation}{Notation and Conventions}

\newtheorem{step}{Step}
\begin{document}

\begin{abstract}
We define an ``ample canonical height" for an endomorphism on a projective variety,
which is essentially a generalization of the canonical heights for polarized endomorphisms 
introduced by Call--Silverman.
We formulate a dynamical analogue of the Northcott finiteness theorem 
for ample canonical heights 
as a conjecture, and prove it for endomorphisms on 
varieties of small Picard numbers, abelian varieties, and surfaces.
As applications, for the endomorphisms which satisfy the conjecture,
we show the non-density of the set of preperiodic points over a fixed number field,
and obtain a dynamical Mordell--Lang type result  
on the intersection of two Zariski dense orbits of two endomorphisms on a common 
variety.
\end{abstract}

\maketitle

\tableofcontents 

\section{Introduction}\label{sec1}
Let $X$ be a smooth projective variety over $\overline{\mathbb Q}$ and $f$ a 
polarized endomorphism, that is, a surjective morphism from $X$ onto $X$ with 
an ample divisor $H$ such that $f^*H \sim dH$ for some $d >1$.
Let $h_H$ be a height associated to $H$.
Then we can define the \textit{canonical height associated to $H$} 
due to Call--Silverman \cite{CaSi93}: 
$$\hat h_{H,f}(x)=\lim_{n \to \infty} \frac{h_H(f^n(x))}{d^n}.$$
Then $\hat h_{H,f}$ is considered as a new ample height on $X$ 
which reflects the dynamics of $f$; 
for example, $\hat h_{H,f} \circ f = d \hat h_{H,f}$ holds and, for any point $x$,
$\hat h_{H,f}(x)=0$ if and only if $x$ is $f$-preperiodic 
i.e.~$\{ x, f(x), f^2(x),\ldots \}$ is finite.
Moreover, the Northcott finiteness theorem (Theorem \ref{thm_Northcott}) implies that 
the set 
$$\{ x \in X(K) \mid  \hat h_{H,f}(x)=0 \}$$
is finite for any number field $K$.
This result might be seen as a dynamical version of the Northcott finiteness theorem.
Eventually,
it follows that the set of $f$-preperiodic $K$-rational points 
is finite for any number field $K$.
In particular, any point $x \in X(\overline{\mathbb Q})$ with $\hat h_{H,f}(x)>0$ 
is not $f$-preperiodic.

Thus canonical height is a powerful tool to study the dynamics of 
polarized endomorphisms over number fields.
So it is nice if we have such a canonical height for general endomorphisms.
But the following example shows that we should modify the definition of 
canonical heights for general endomorphisms.

\begin{ex}\label{ex1.1}
Let $E$ be an elliptic curve over $\overline{\mathbb Q}$ and set 
$X=E \times E$.
Take two integers $a \in \mathbb Z_{\geq 2}$, $b \in \mathbb Z \setminus \{0\}$ and 
let $f$ be an endomorphism on $X$ defined as 
$f(x,y)=(ax+by,ay)$ for $(x,y) \in X(\overline{\mathbb Q})$.
Then we have $f^n(x,y)=(a^nx+na^{n-1}by,a^ny)$.
Let $\hat h_E$, $\hat h_X$ be N\'eron--Tate heights on $E,X$ respectively 
such that $\hat h_X(x,y)=\hat h_E(x)+\hat h_E(y)$.
Then $\hat h_X(f^n(0,y))=\hat h_E(na^{n-1}by)+\hat h_E(a^ny)
=a^{2n}(n^2 a^{-2}b^2+1) \hat h_E(y)$.
If $y \in E(\overline{\mathbb Q})$ is not a torsion point, then $\hat h_E(y)>0$ and so 
$\hat h_X(f^n(0,y))$ grows like $a^{2n}n^2$ as $n$ grows.
In this case, we should define the canonical height $\hat h_f(x,y)$ 
with respect to $f$ at $(0,y)$ as 
$$\hat h_f(0,y)
=\lim_{n \to \infty} \frac{\hat h_X(f^n(0,y))}{a^{2n}n^2}
=a^{-2} b^2 \hat h_E(y).$$
The dynamical degree $\delta_f$ of $f$ (cf.~Notation and Conventions below) 
is equal to $a^2$ (cf.~Theorem \ref{thm5.2}).
\end{ex}

Taking this example into account, we will define ample canonical heights for 
general endomorphisms.
Silverman \cite[p.~649]{Sil14} defined the (upper) canonical heights for 
rational self-maps on projective spaces as follows: 
let $\varphi: \mathbb P^d \dashrightarrow \mathbb P^d$ be a rational map with 
$$\delta_{\varphi}=\lim_{n \to \infty} (\deg(\varphi^n))^{1/n} >1.$$
Then the upper canonical height at $P \in \mathbb P^d(\overline{\mathbb Q})$ is 
$$\hat h_{\varphi}(P)=\limsup_{n \to \infty} 
\frac{h(\varphi^n(P))}{n^{l_{\varphi}}\delta_{\varphi}^n},$$
where $h$ is the natural height function on $\mathbb P^d$ and 
$$l_{\varphi}=\inf \left\{ l \geq 0 
\mid \sup_{n \geq 1} \frac{\deg(\varphi^n)}{n^l \delta_{\varphi}^n} < \infty \right\}.$$
Note that we may have $\hat h_{\varphi}(P) = \infty$ for some 
rational self-map $\varphi$ and $P \in \mathbb P^d(\overline{\mathbb Q})$.

Modifying the definition of the canonical height for a self-map on a projective space,
we define (upper/lower) canonical height for (not necessarily polarized) endomorphisms.
For a pair $(X,f)$ of a projective variety $X$ over $\overline{\mathbb Q}$ 
and an endomorphism $f$ on $X$,
fix an ample height $h_X \geq 1$ i.e.~a height associated to an ample divisor on $X$.
Let $\delta_f$ be the (first) dynamical degree of $f$ 
(see Notation and Conventions below), and 
$l_f$ the minimal non-negative integer such that 
the sequence $\{ \delta_f^{-n} n^{-l_f} h_X(f^n(x)) \}_{n=0}^\infty$ is upper bounded 
for every $x \in X(\overline{\mathbb Q})$.
The existence of such $l_f$ is proved 
by Matsuzawa \cite[Theorem 1.6]{Mat16} (cf.~Theorem 2.6).
We define the \textit{upper} (resp.~\textit{lower}) \textit{ample canonical height} 
$\overline h_f, \underline h_f$ as 
$$
\overline h_f(x)=\limsup_{n \to \infty} \frac{h_X(f^n(x))}{\delta_f^n n^{l_f}},
$$
$$
\underline h_f(x)=\liminf_{n \to \infty} \frac{h_X(f^n(x))}{\delta_f^n n^{l_f}}.
$$

It is obvious by definition that $\overline h_f$ and $\underline h_f$ take  
finite and non-negative values at every point.
If $f$ is a polarized endomorphism, then $\overline h_f, \underline h_f$ are 
essentially equivalent to the canonical height of Call--Silverman, 
as we will see in Section \ref{sec4}.
So we can regard the notion of ample canonical heights as a generalization of 
canonical heights for polarized endomorphisms.

On the other hand, Kawaguchi and Silverman \cite{KaSi16a} introduced the canonical height 
$\hat h_{D,f}$ associated to a \textit{nef} $\mathbb R$-Cartier $\mathbb R$-divosor $D$ 
such that $D$ is not numerically trivial and $f^*D$ is numerically equivalent to $\delta_fD$,
which we call a \textit{nef canonical height} (cf.~Definition \ref{defn_nefcnht}).
Note that such $D$ always exists 
due to Perron--Frobenius--Birkhoff theorem (Theorem \ref{thm_pfb}).
Then the following questions naturally arise.

\begin{ques}\label{ques1}
Let $X$ be a smooth projective variety over $\overline{\mathbb Q}$ and $f$ an endomorphism on $X$ with $\delta_f>1$.
\begin{itemize}
\item[(i)] Whether $\overline h_f \asymp \underline h_f$ holds or not?
\item[(ii)] Does there exist  a nef canonical height $\hat h_{D,f}$ such that 
$\overline h_f \asymp \hat h_{D,f}$ and $\underline h_f \asymp \hat h_{D,f}$?
\end{itemize}
\end{ques}

For the definition of the relation ``$\asymp$", see Notation and Conventions below.
We will see that Question \ref{ques1} has positive answers in the following cases.
\begin{itemize}
\item There is an ample $\mathbb R$-divisor $H$ such that $f^*H \equiv \delta_f H$ 
(Theorem \ref{thm4.1} (i)).
\item The Picard number of $X$ is two and $f$ is an automorphism 
(Theorem \ref{thm4.2} (ii)).
\item $X$ is a Calabi--Yau threefold whose Picard number is at most three and 
$f$ is an automorphism (Theorem \ref{thm4.5} (i)).
\item $X$ is a surface and $f$ is an automorphism (Theorem \ref{thm6.6} (i)).
\end{itemize}
But in general the relationship of these height functions is not clear at the moment.

We expect that ample canonical heights have nice properties reflecting the 
dynamics of $f$.
As an analogy with the Northcott finiteness theorem for ample heights,
the set of points at which the (lower) ample canonical height vanishes should be 
``small".
Indeed, the zero set of the canonical height  
for a polarized endomorphism is ``small" as we saw above.

Let $K \subset \overline{\mathbb Q}$ be any subfield.
The symbol $Z_f(K)$ denotes 
the set of $K$-rational points of $X$ at which $\underline h_f$ takes zero.
The main objective of this article is to study the structure of $Z_f(K)$.
For that, we give the following conjecture as a dynamical analogue of 
the Northcott finiteness theorem.

\begin{conj}\label{conj_zf}
Let $X$ be a smooth projective variety over $\overline{\mathbb Q}$ 
and $f$ an endomorphism on $X$ with $\delta_f>1$.
Take any number field $K$.
Then there is an $f$-invariant proper closed subset $V$ of $X$ including $Z_f(K)$.
\end{conj}

Clearly, it is sufficient for proving Conjecture \ref{conj_zf} 
to show the existence of such a closed subset 
for any sufficiently large number field.
The assumption that $\delta_f >1$ is necessary (see Example \ref{ex3.4.2} below).

We make a weaker conjecture,
which is a generalization of a conjecture of Kawaguchi and Silverman 
\cite[Conjecture 6 (d)]{KaSi16a}
restricted to the endomorphism case (cf.~Proposition \ref{prop3.2} (iv)).

\begin{conj}\label{conj_main}
Let $X$ be a smooth projective variety over $\overline{\mathbb Q}$ 
and $f$ an endomorphism on $X$ with $\delta_f>1$.
For any point $x \in X(\overline{\mathbb Q})$ whose  
$f$-orbit 
$$O_f(x)=\{ x, f(x),f^2(x),\ldots \}$$
is dense in Zariski topology, we have $\underline h_f(x)>0$.
\end{conj}

Let $f$ be an endomorphism on a smooth projective variety $X$ with $\delta_f>1$ and 
assume that Conjecture \ref{conj_zf} holds for $f$.
Let $x \in X(\overline{\mathbb Q})$ be a point such that 
$O_f(x)$ is dense.
Here $O_f(x)$ is contained in $X(K)$ for a sufficiently large number field 
$K \subset \overline{\mathbb Q}$.
Suppose $x \in Z_f(K)$.
Then $O_f(x) \subset  Z_f(K)$ since $f(Z_f(K)) \subset Z_f(K)$ (cf.~Proposition \ref{prop3.2}),
but this contradicts Conjecture \ref{conj_zf} for $f$.
Hence $x \not\in Z_f(K)$ i.e.~$\underline h_f(x)>0$.
Thus Conjecture \ref{conj_zf} implies Conjecture \ref{conj_main}.

Our aim in this article is to show that Conjecture \ref{conj_zf} holds for 
certain endomorphisms.
The main result is the following.

\begin{thm}\label{thm_main}
Let $X$ be a smooth projective variety and $f$ an endomorphism on $X$ with $\delta_f>1$.
Then Conjecture \ref{conj_zf} holds in the following situations.
\begin{itemize}
\item  {\rm (Theorem \ref{thm4.1})} 
$f^*H \equiv \delta_f H$ for an ample $\mathbb R$-divisor $H$ on $X$.
This contains the case when the Picard number of $X$ is one.
\item {\rm (Theorem \ref{thm4.2})} 
$\rho(X) \leq 2$ and $f$ is an automorphism.
\item {\rm (Theorem \ref{thm5.1})}
$X$ is an abelian variety.
\item {\rm (Theorem \ref{thm6.6} and Theorem \ref{thm7.7})}
$X$ is a smooth projective surface.
\end{itemize}
\end{thm}

Let us briefly see how these results are proved.
The first case is easily shown 
because the ample canonical height for a polarized endomorphism 
is equivalent to the canonical height due to Call--Silverman.

For the $\rho(X)=2$ case, we can take two nef $\mathbb R$-divisors $D_{\pm}$ which are 
eigenvectors of $f^*$ in $N^1(X)_{\mathbb R}$ and the associated canonical heights 
$\hat h_{D_{\pm},f}$, which help us to compute the ample canonical height.

If $X$ is an abelian variety and $f \in \End(X)$, then $\{f^n \}_{n=0}^\infty$ satisfies  
a $\mathbb Q$-linear recurrence relation in $\End(X)_{\mathbb Q}$.
Then we can compute ample canonical heights 
with the aid of the recurrence relation.

Surface automorphism case follows from arguments due to Kawaguchi 
 \cite{Kaw08} and Kawaguchi--Silverman \cite{KaSi14}.
We take two nef canonical heights $\hat h^{\pm}$ for $f^{\pm}$.
Then it turns out that 
$\hat h^{+}$ is equivalent to the ample canonical height.
So the assertion follows from results in \cite{Kaw08}.
Surface endomorphism case is proved by using the results due to Matsuzawa, Sano, and 
the author in \cite{MSS17}.
In \cite{MSS17}, it is proved that any non-automorphic endomorphism on 
a minimal surface which is isomorphic to neither $\mathbb P^2$ nor abelian surfaces 
admits a certain fibration to a curve.
Then some comutation of height on the surface 
is reduced to computation of a height on the curve.

\begin{rem}\label{rem_rat}
It is not clear that the above definition 
of ample canonical height works for an arbitrary rational self-map $f: X \dashrightarrow X$ 
because we do not know whether a non-negative integer $l$ which makes
$\{ \delta_f^{-n} n^{-l} h_X(f^n(x)) \}_{n=0}^\infty$ bounded for every $x$ 
exists or not.
\end{rem}

There are already various constructions of ``canonical heights" for certain self-maps.
Here the term ``canonical heights" means functions which are constructed from 
a (ample or nef) height function and reflect some dynamical behavior of the self-map.
So the definition of canonical heights in the following references are different in general. 
As mentioned above, Call and Silverman \cite{CaSi93} defined canonical heights 
for polarized endomorphisms, which includes the N\'eron--Tate heights on abelian varieties 
as a special case.
Kawaguchi \cite{Kaw06}, \cite{Kaw13} and Lee \cite{Lee13} constructed canonical heights 
for regular polynomial automorphisms.
Kawaguchi \cite{Kaw08} constructed canonical heights for surface automorphisms.
Siverman \cite{Sil14} defined and studied canonical heights for rational self-maps on 
projective spaces.
Kawaguchi and Silverman \cite{KaSi16a} showed that there always exists 
a nef canonical height, that is, a canonical height 
associated to a nef $\mathbb R$-divisor 
for any endomorphisms on normal projective varieties 
(cf.~Theorem \ref{thm_pfb} and Theorem \ref{thm_canht}).
Jonsson and Wulcan \cite{JoWu12} constructed canonical heights for plane polynomial 
maps of small topological degree.
Jonsson and Reschke \cite{JoRe15} constructed canonical heights for birational self-maps 
on surfaces.

This paper proceeds as follows.
In Section \ref{sec2}, we recall fundamental facts on heights.
In Section \ref{sec3}, we define ample canonical heights again and 
show elementary properties of them.
Section \ref{sec4} treats endomorphisms on smooth projective 
varieties of small Picard numbers.
Endomorphisms on smooth projective varieties of Picard number one,
automorphisms on smooth projective varieties of Picard number $\leq 2$,
and automorphisms on Calabi--Yau threefolds of Picard number $\leq 3$ are mainly studied.
We investigate endomorphisms on abelian varieties in Section \ref{sec5},
and endomorphisms on surfaces in Section \ref{sec6} and Section \ref{sec7}.
In Section \ref{sec8}, we make two applications of Theorem \ref{thm_main}.
First, we see that Conjecture \ref{conj_zf} implies the non-density of 
the preperiodic points over any fixed number field (Proposition \ref{prop8.0}),
and then we obtain such a non-density result for endomorphisms 
appearing in Theorem \ref{thm_main} (Theorem \ref{thm8.0.1}).
Second, we describe the intersection of two dense orbits $O_f(x), O_g(y)$ of 
endomorphisms $f,g$ on a variety.
Main results in Section \ref{sec8} are stated without the notion of height.
\begin{notation}
\begin{itemize}
\item[]
\item Throughout this article, we work over $\overline{\mathbb Q}$,
the algebraic closure of the rational number field.

\item A \textit{curve} (resp.~\textit{surface}) simply means a smooth projective variety 
of dimension one (resp.~dimension two) unless otherwise stated.

\item Let $X$ be a projective variety.
An \textit{endomorphism on $X$} means a surjective morphism from $X$ onto $X$.
A \textit{non-trivial endomorphism on $X$} means an endomorphism on $X$ 
which is not an 
automorphism.

\item Let $X$ be an abelian variety.
$\End(X)$ denotes the set of (not necessarily surjective) 
algebraic group homomorphisms from $X$ to $X$.
Set $\End(E)_{\mathbb K}=\End(E) \otimes_{\mathbb Z} \mathbb K$ 
($\mathbb K=\mathbb Q, \mathbb R$ or $\mathbb C$).

\item Let $X$ be a projective variety and $f$ an endomorphism on $X$.
\begin{itemize}
\item[(i)]
The (\textit{forward}) \textit{$f$-orbit} of a point $x \in X(\overline{\mathbb Q})$ is 
the set $O_f(x) = \{ x, f(x), f^2(x), \ldots \}$.
\item[(ii)]
A point $x \in X(\overline{\mathbb Q})$ is \textit{$f$-periodic} if 
$f^n(x)=x$ for a positive integer $n$.
For any subfield $K \subset \overline{\mathbb Q}$,
$\Per_f(K)$ denotes the set of $f$-periodic $K$-rational points of $X$.
\item[(iii)]
A point $x \in X(\overline{\mathbb Q})$ is \textit{$f$-preperiodic} if 
$f^k(x)$ is $f$-periodic for a positive integer $k$.
For any subfield $K \subset \overline{\mathbb Q}$,
$\Preper_f(K)$ denotes the set of $f$-preperiodic $K$-rational points of $X$.

It is clear that $x$ is $f$-preperiodic if and only if $O_f(x)$ is finite.
Moreover, if $f$ is an automorphism, then $x$ is $f$-preperiodic if and only if 
$x$ is $f$-periodic.
\item[(iv)]
A closed subset $V \subset X$ is \textit{$f$-invariant} if $f(V) \subset V$,
and \textit{$f$-periodic} if it is $f^N$-invariant for some positive integer $N$.
\end{itemize}

\item Let $X$ be a smooth projective variety and $f$ an endomorphism on $X$.
Take an ample divisor $H$ on $X$. Then the limit 
$$\delta_f= \lim_{n \to \infty} ((f^n)^*H \cdot H^{\dim X -1})^{1/n}$$
exists and is independent of the choice of $H$.
The invariant $\delta_f$ is called the (\textit{first}) \textit{dynamical degree of $f$}.

\item Let $\mathbb K$ be $\mathbb R$ or $\mathbb C$.
For a $\mathbb K$-linear endomorphism $f: V \to V$ on a $\mathbb K$-vector space $V$,
$\rho(f)$ denotes the spectral radius of $f$, that is, 
the maximum of absolute values of eigenvalues of $f$.

\item The symbols $\sim$ (resp.~$\sim_{\mathbb Q}$, $\sim_{\mathbb R}$) and 
$\equiv$ mean 
the linear equivalence (resp.~$\mathbb Q$-linear equivalence,  
$\mathbb R$-linear equivalence) and the numerical equivalence on divisors.

\item For a projective variety $X$, $N^1(X)$ denotes 
the abelian group of the numerical equivalence classes of 
Cartier divisors of $X$. 
Set $N^1(X)_{\mathbb R}=N^1(X) \otimes_{\mathbb Z} \mathbb R$ and 
$\rho(X)=\dim_{\mathbb R} N^1(X)_{\mathbb R}$.
The number $\rho(X)$ is called the \textit{Picard number of $X$}.

\item Let $h_1, h_2$ be non-negative functions on a same domain.
We say that \textit{$h_2$ dominates $h_1$}, denoted by $h_1 \prec h_2$,
if there is a positive constant $C$ such that $h_1 \leq Ch_2$.
We say that $h_1$ is \textit{equivalent} to $h_2$,
denoted by $h_1 \asymp h_2$, if $h_1 \prec h_2$ and $h_2 \prec h_1$.

\item Let $f$, $g$ and $h$ be  $\mathbb R$-valued functions on a domain.
The equality $f = g + O(h)$ means that there is a positive constant $C$ such that 
$|f-g| \leq C |h|$.
In particular,
the equality $f=g + O(1)$ means that there is a positive constant $C$ such that 
$|f-g| \leq C$.

\item Let $X$ be a projective variety.
For an $\mathbb R$-Cartier $\mathbb R$-divisor $D$ on $X$,
a function $h_D: X(\overline{\mathbb Q}) \to \mathbb R$ is determined up to the 
difference of a bounded function.
$h_D$ is called the \textit{height function associated to $D$}.
For definition and properties of height functions, see e.g.~\cite[Part B]{HiSi00} or \cite[Chapter 3]{Lan83}.

\item For a projective variety $X$, we always fix an ample height function $h_X$,
that is, a height function associated to an ample divisor, with $h_X \geq 1$.
If $h_1, h_2$ are ample height functions on $X$ with $h_1, h_2 \geq 1$,
then $h_1 \asymp h_2$ (cf.~Lemma \ref{lem2.1}).

\item Let $X$ be a normal projective variety and $f$ an endomorphism on $X$.
Then the limit 
$$\alpha_f(x)=\lim_{n \to \infty} h_X(f^n(x))^{1/n}$$
exists and is independent of the choice of $h_X$ for every $x \in X(\overline{\mathbb Q})$ 
(\cite[Theorem 3 (a)]{KaSi16b}).
The number $\alpha_f(x)$ is called the \textit{arithmetic degree of $f$ at $x$}.
For details, see \cite{KaSi16a} and \cite{KaSi16b}.

\end{itemize}
\end{notation}

\begin{ack}
I would like to thank Professors Osamu Fujino and Mattias Jonsson 
for valuable comments.
I am grateful to Kaoru Sano and  Yohsuke Matsuzawa 
for insightful discussions and comments,
and Kenta Hashizume for answering some questions.
I appreciate the referee giving me so many important comments.
\end{ack}

\section{Basic results on heights}\label{sec2}

In this section, we recall some basic results on heights which are used later.

\begin{lem}\label{lem2.1}
Let $X$ be a projective variety.
For any $\mathbb R$-divisor $D$ on $X$,
$h_D \prec h_X$.
\end{lem}

\begin{proof}
We set $h_X=h_H \geq 1$ for some ample divisor $H$ on $X$.
Take a sufficiently large integer $N$ such that $NH-D$ is ample.
Then $h_H = (1/N)h_{NH} +O(1)=(1/N)(h_D+h_{NH-D}) +O(1) \geq (1/N)h_D+O(1)$.
So $h_D \leq N h_H + O(1)$.
Since $h_H \geq 1$, we can take a sufficiently large $C>0$ such that $h_D \leq C h_H$.
\end{proof}

\begin{thm}[Northcott finiteness theorem]\label{thm_Northcott}
Let $X$ be a projective variety over a number field $K$, 
$H$ an ample $\mathbb R$-Cartier 
$\mathbb R$-divisor on $X$,
$d$ a positive integer,
and $B$ a positive constant.
Then the set 
$$\{ x \in X(L) \mid L \mathrm{\ is\ a\ number\ field\ with\ } 
[L: K] \leq d,\ h_H(x) \leq B \}$$ 
is finite.
\end{thm} 

From the Northcott finiteness theorem, we can deduce a similar result for semiample 
divisors.

\begin{cor}\label{cor_Northcott}
Let $X$ be a projective variety over a number field $K$, 
$D$ a semiample 
Cartier divisor on $X$ with $D \not\sim 0$,
$d$ a positive integer,
and $B$ a positive constant.
Then the set 
$$\{ x \in X(L) \mid L \mathrm{\ is\ a\ number\ field\ with\ } 
[L: K] \leq d,\ h_D(x) \leq B \}$$ 
is not dense.
\end{cor}

\begin{proof}
Take a sufficiently large integer $N$ such that $ND$ is base point free.
Then there is a surjective morphism $\phi: X \to Y$ to a projective variety $Y$ 
such that $ND \sim \phi^*H$ for some ample divisor $H$ on $Y$.
Then $h_H \circ \phi = Nh_D+ O(1)$, so we can take $C>0$ such that 
$h_H \circ \phi \leq Nh_D + C$.
Set 
$$S=\{ x \in X(L) \mid L \mathrm{\ is\ a\ number\ field\ with\ } 
[L: K] \leq d,\ h_D(x) \leq B \},$$
$$T=\{ y \in Y(L) \mid L \mathrm{\ is\ a\ number\ field\ with\ } 
[L:K] \leq d,\ h_H(y) \leq NB+C \}.$$
Then $S \subset \phi^{-1}(T)$, and $T$ is a finite set by Theorem \ref{thm_Northcott}.
So $S$ is contained in a proper closed subset.
\end{proof}

As an application of Perron--Frobenius--Birkhoff theorem, we obtain the following 
(cf.~{\cite[Remark 31]{KaSi16a}}).

\begin{thm}\label{thm_pfb}
Let $X$ be a projective variety and $f$ an endomorphism on $X$.
Then there is a nef $\mathbb R$-Cartier $\mathbb R$-divisor $D$ on $X$ such that 
$D \not\equiv 0$ and $f^*D \equiv \delta_fD$.
\end{thm}

For an $\mathbb R$-Cartier $\mathbb R$-divisor $D$ which is an eigenvector of 
$f^*: N^1(X)_{\mathbb R} \to N^1(X)_{\mathbb R}$,
we can define the \textit{canonical height associated to $D$} 
under some assumptions.

\begin{thm}[{\cite{CaSi93}} and {\cite[Theorem 5]{KaSi16a}}]\label{thm_canht}
Let $X$ be a projective variety, $f$ an endomorphism on $X$ with $\delta_f >1$, 
and $D$ an $\mathbb R$-Cartier $\mathbb R$-divisor on $X$.
\begin{itemize}
\item[(i)]
Assume that $f^*D \sim_{\mathbb R} \lambda D$ with $\lambda >1$.
Then the limit 
$$\hat h_{D, f}(x)=\lim_{n \to \infty} \frac{h_D(f^n(x))}{\lambda^n}$$
exists for every $x \in X(\overline{\mathbb Q})$ and satisfies
$\hat h_{D,f} = h_D +O(1)$.
\item[(ii)]
Assume that $f^*D \equiv \lambda D$ with $\lambda > \sqrt{\delta_f}$.
Then the limit 
$$\hat h_{D, f}(x)=\lim_{n \to \infty} \frac{h_D(f^n(x))}{\lambda^n}$$
exists for every $x \in X(\overline{\mathbb Q})$ and satisfies
$\hat h_{D,f} = h_D +O(\sqrt{h_X})$.
\end{itemize}
\end{thm}

\begin{defn}\label{defn_nefcnht}
Let $X$ be a projective variety and $f$ an endomorphism on $X$ with $\delta_f>1$.
Take a nef $\mathbb R$-Cartier $\mathbb R$-divisor $D$ on $X$ such that 
$D \not\equiv 0$ and $f^*D \equiv \delta_fD$ by using Theorem \ref{thm_pfb}.
Then Theorem \ref{thm_canht} implies that the limit 
$$\hat h_{D, f}(x)=\lim_{n \to \infty} \frac{h_D(f^n(x))}{\delta_f^n}$$
exists for every $x \in X(\overline{\mathbb Q})$.
We call $\hat h_{D,f}$ a \textit{nef canonical height for $f$ associated to $D$}.
\end{defn}

To estimate the growth of heights $h_X(f^n(x))$ as $n$ increases,
the following result is fundamental.

\begin{thm}[{\cite[Theorem 1.6]{Mat16}}]\label{thm_Mat}
Let $X$ be a projective variety with $\rho(X)=r$ 
and $f$ an endomorphism on $X$.
\begin{itemize}
\item[(i)] 
If $\delta_f=1$, then there is a positive constant $C>0$ such that 
$h_X \circ f^n \leq C n^{2r+2} h_X$ for every $n \in \mathbb Z_{\geq 0}$.
\item[(ii)]
If $\delta_f >1$, then there is a positive constant $C>0$ such that 
$h_X \circ f^n \leq C \delta_f^n n^r h_X$ for every $n \in \mathbb Z_{\geq 0}$.
\end{itemize}
\end{thm}

Theorem \ref{thm_Mat} deduces the following weaker inequality.
Note that the inequality is proved for dominant rational self-maps, which is not needed 
in this article.

\begin{thm}[{\cite[Theorem 26]{KaSi16a}}, {\cite[Theorem 1.4]{Mat16}}]\label{thm_KSM}
Let $X$ be a projective variety, $f$ an endomorphism on $X$ with $\delta_f >1$, 
and $\varepsilon>0$ any positive constant.
Then there is a positive constant $C>0$ such that 
$h_X \circ f^n \leq C (\delta_f+\varepsilon)^n h_X$ for every $n \in \mathbb Z_{\geq 0}$.
\end{thm}

Theorem \ref{thm_KSM} implies that $\alpha_f(x) \leq \delta_f$ for every point $x$.
On the other hand,
any dynamical system $(X,f)$ has a point whose arithmetic degree attains the dynamical 
degree:

\begin{thm}[{\cite[Theorem 1.6]{MSS17}}]\label{thm_existence}
Let $X$ be a smooth projective variety and $f$ an endomorphism on $X$.
Then there is a point $x \in X(\overline{\mathbb Q})$ such that $\alpha_f(x)=\delta_f$.
\end{thm}


\section{Ample canonical heights}\label{sec3}

In this section, we will define ample canonical heights for endomorphisms 
and prove some elementary properties.
In what follows, we always assume the smoothness of projective varieties for simplicity.

We define ample canonical heights as follows.

\begin{defn}\label{defn3.1}
Let $X$ be a smooth projective variety and $f$ an endomorphism on $X$.
\begin{itemize}
\item[(i)]
Let $l_f$ be the smallest non-negative integer such that the sequence 
$$\left\{ \frac{h_X(f^n(x))}{\delta_f^n n^{l_f}} \right\}_{n=0}^\infty$$
 is upper bounded for every 
$x \in X(\overline{\mathbb Q})$.
Theorem \ref{thm_Mat} guarantees the existence of such $l_f$.
\item[(ii)]
Set 
$$\overline h_f(x) = \limsup_{n \to \infty}  \frac{h_X(f^n(x))}{\delta_f^{n} n^{l_f}},\ \ 
\underline h_f(x) = \liminf_{n \to \infty} \frac{h_X(f^n(x))}{\delta_f^{n} n^{l_f}},$$
which we call \textit{upper ample canonical height for $f$},
\textit{lower ample canonical height for $f$},
respectively.
\end{itemize}
\end{defn}

\begin{rem}\label{rem3.1.1}
$\overline h_f$ and $\underline h_f$ depend on the choice of $h_X$.
If $\overline{h}'_f, \underline{h}'_f$ are upper and lower ample canonical heights 
associated to another ample height $h'_X$, then it is clear that 
$\overline h_f \asymp \overline{h}'_f$ and $\underline h_f \asymp \underline{h}'_f$.
In particular, the condition that $\overline h_f(x)=0$ (resp.~$\underline h_f(x)=0$)
is independent of the choice of $h_X$.
\end{rem}

\begin{defn}\label{defn3.1.2}
Let $X$ be a smooth projective variety and $f$ an endomorphism on $X$.
For any subfield $K \subset \overline{\mathbb Q}$, we set 
$$Z_f(K)=\{ x \in X(K) \mid \underline h_f(x)=0 \}.$$
\end{defn}

As we saw in Remark \ref{rem3.1.1}, $Z_f(K)$ is independent 
of the choice of $h_X$.
Proposition \ref{prop3.2} (iii) below shows that 
$Z_f(\overline{\mathbb Q})$ is an $f$-invariant subset i.e.~
$f(Z_f(\overline{\mathbb Q})) \subset Z_f(\overline{\mathbb Q})$.

\begin{ex}\label{ex3.4.1}
Let $X$ be a smooth projective variety and $f: X \to X$ a polarized endomorphism: 
$f^*H \sim dH$ for an ample divisor $H$ and $d >1$.
Then it follows that $\delta_f=d$.
Take a height $h_H$ associated to $H$ as satisfying $h_H \geq 1$.
Then $h_H \asymp h_X$.
So 
$$\limsup_{n \to \infty} d^{-n} h_X(f^n(x)) 
\asymp \lim_{n \to \infty} d^{-n} h_H(f^n(x))= \hat h_{H,f}(x).$$
This implies that $l_f=0$ and $\overline h_f \asymp \hat h_{H,f}$.
Similarly $\underline h_f \asymp \hat h_{H,f}$.
Thus $\overline h_f, \underline h_f$ are essentially equivalent to the canonical 
height $\hat h_{H,f}$ associated to $H$.
It follows that $Z_f(\overline{\mathbb Q})$ is the set of $f$-preperiodic points.
We will show a more general result in Theorem \ref{thm4.1}.
\end{ex}

\begin{ex}\label{ex3.4.2}
On the other hand, let 
$X$ be a smooth projective variety and $f: X \to X$ an endomorphism such that 
$f^n \neq \id_X$ for every $n \in \mathbb Z_{>0}$ and 
$f^*H \sim H$ for an ample divisor $H$ (e.g.~automorphisms of infinite order 
on projective spaces).
Then we have $\delta_f=1$.
As before,
take a height $h_H$ associated to $H$ as satisfying $h_H \geq 1$.
Since $h_H \circ f=h_H+O(1)$, we can take $C>0$ such that 
$|h_H \circ f -h_H| \leq C$.
Then 
$$|h_H \circ f^n| \leq \sum_{k=1}^n |h_H \circ f^k- h_H \circ f^{k-1}| +h_H
\leq nC+h_H.$$
Hence $\limsup_n n^{-1} h_H(f^n(x))<\infty$ and so $l_f \leq 1$.
On the other hand,
we can take a non-$f$-preperiodic point $x$ (cf.~\cite{Ame11}),
and then $\{ h_H(f^n(x)) \}_{n=0}^\infty$ 
is not upper bounded by the Northcott finiteness theorem.
So $l_f=1$.

Set $X=\mathbb P^1$ and $f(x:y)=(x+y:y)$.
Then $f^n(x:y)=(x+ny:y)$.
Fix a number field $K$ and take any point $P=(x:y) \in X(K)$.
Let $h$ be the usual height function on $X$ (cf.~\cite[B.2]{HiSi00}).
Then 
\begin{align*}
h(f^n(P))
&= \sum_{v \in M_K} \log \max \{| |x+ny||_v, ||y||_v \} \\
&\leq \sum_{v \in M_K} \log \max \{||x||_v, ||ny||_v, ||y||_v \}.
\end{align*}
Here $\lim_n n^{-1}\log ||ny||_v =\lim_n n^{-1} (\log n+ \log ||y||_v)=0$,
so $\overline h_f(P)= \limsup_n n^{-1} h(f^n(P))=0$.
Since $K$ and $P$ are arbitrary, $\overline h_f=\underline h_f=0$ and so 
$Z_f(\overline{\mathbb Q})=X(\overline{\mathbb Q})$.
Hence Conjecture \ref{conj_zf} and Conjecture \ref{conj_main} fail for $f$.

This example suggests that ample canonical heights do not work well 
for endomorphisms with 
dynamical degree one, or at least we should modify the definition of ample canonical heights 
for such endomorphisms.
\end{ex}

From now on, we will show some elementary results on ample canonical heights.
The following proposition is similar to \cite[Proposition 19]{Sil14}.

\begin{prop}\label{prop3.2}
Let $X$ be a smooth projective variety and $f$ an endomorphism on $X$.
\begin{itemize}
\item[(i)]
$\overline h_f$ and $\underline h_f$ are non-negative $\mathbb R$-valued functions.
\item[(ii)]
Assume that $\delta_f>1$ or $l_f>0$. Then 
$\overline h_f(x)=0$ for any $f$-preperiodic point $x \in X(\overline{\mathbb Q})$.
\item[(iii)]
$\overline h_f \circ f = \delta_f \overline h_f$, 
$\underline h_f \circ f = \delta_f \underline h_f$.
\item[(iv)]
For $x \in X(\overline{\mathbb Q})$, assume that $\overline h_f(x) >0$.
Then $\alpha_f(x)=\delta_f$.
\end{itemize}
\end{prop}

\begin{proof}
(i) and (ii) are clear by definition.

(iii) Take any $x \in X(\overline{\mathbb Q})$. Then 
\begin{align*}
\overline h_f(f(x))
&=\limsup_{n \to \infty} \frac{h_X(f^{n+1}(x))}{\delta_f^n n^{l_f}} \\
&=\limsup_{n \to \infty} \delta_f \left(1+\frac{1}{n} \right)^{l_f}
\frac{h_X(f^{n+1}(x))}{\delta_f^{n+1}(n+1)^{l_f}} \\
&= \delta_f \overline h_f(x).
\end{align*}
Similarly $\underline h_f(f(x))=\delta_f \underline h_f(x)$.

(iv) We compute 
$$\overline h_f(x)= \limsup_{n \to \infty} \frac{h_X(f^n(x))}{\delta_f^n n^{l_f}}
= \limsup_{n \to \infty} \left( \frac{h_X(f^n(x))^{1/n}}{\delta_f n^{l_f/n}} \right)^n.$$
Now it follows that 
$$\lim_{n \to \infty} \frac{h_X(f^n(x))^{1/n}}{\delta_f n^{l_f/n}} 
= \frac{\alpha_f(x)}{\delta_f}.$$
So $\alpha_f(x) < \delta_f$ implies that $\overline h_f(x) =0$.
\end{proof}

\begin{lem}\label{lem3.5}
Let $X$ be a smooth projective variety and $f$ an endomorphism on $X$.
Take a positive integer $N$.
\begin{itemize}
\item[(i)]
$\delta_{f^N}=\delta_f^N$, $l_{f^N}=l_f$.
\item[(ii)]
Let $K \subset \overline{\mathbb Q}$ be a subfield where $X$ and $f$ are defined.
Then 
$Z_f(K)=\bigcup_{i=0}^{N-1} (f^i)^{-1}(Z_{f^N}(K))$.
\item[(iii)]
Conjecture \ref{conj_zf} holds for $f$ if and only if it holds for $f^N$.
\end{itemize}
\end{lem}

\begin{proof}
(i) Take an ample divisor $H$ on $X$. Then 
\begin{align*}
\delta_{f^N}
&=\lim_{n \to \infty} ((f^{Nn})^*H \cdot H^{\dim X -1})^{1/n} \\
&=\lim_{n \to \infty} (((f^{Nn})^*H \cdot H^{\dim X -1})^{1/Nn})^N \\
&=\delta_f^N.
\end{align*}
Take any non-negative integer $l$. Set 
$$A_n^{(l)}(x)=  \frac{h_X(f^{n}(x))}{\delta_{f}^n n^l},\ 
B_n^{(l)}(x)= \frac{h_X(f^{Nn}(x))}{\delta_{f^N}^n n^l}.$$
Then 
$$B_n^{(l)}(f^k(x))= \frac{(Nn+k)^l}{n^l} 
\frac{h_X(f^{Nn+k}(x))}{\delta_{f}^{Nn} (Nn+k)^l}
=\left( N+\frac{k}{n} \right)^l \delta_f^{-k} A_{Nn+k}^{(l)}(x).$$
So 
$\{ A_n^{(l)}(x) \}_{n=0}^\infty$ is upper bounded if and only if 
$\{B_n^{(l)}(f^k(x)) \}_{n=0}^\infty$ is upper bounded for every $k \in \{0,1,\ldots,N-1\}$.
This implies that $l_{f^N}=l_f$.

(ii) Set $A_n(x)=A_n^{(l_f)}(x)$ and $B_n(x)=B_n^{(l_f)}(x)$.
The above calculation also shows that 
$\underline h_f(x)=\liminf_n A_n(x)>0$ if and only if 
$\underline h_{f^N}(f^k(x))=\liminf_n B_n(f^k(x))>0$ for every 
$k \in \{0,1,\ldots,N-1\}$.
So the assertion follows.

(iii) Let $K$ be any number field where $X$ and $f$ are defined.
If Conjecture \ref{conj_zf} holds for $f$, we can take an $f$-invariant 
proper closed subset $V \subset X$ such that $Z_f(K) \subset V(K)$.
Then $Z_{f^N}(K) \subset Z_f(K) \subset V(K)$ by (ii), and 
$V$ is clearly $f^N$-invariant.
So Conjecture \ref{conj_zf} holds for $f^N$.

Conversely, assume that Conjecture \ref{conj_zf} holds for $f^N$.
Take an $f^N$-invariant proper closed subset $W \subset X$ such that 
$Z_{f^N}(K) \subset W(K)$.
Then $Z_f(K) \subset \bigcup_{i=0}^{N-1} (f^i)^{-1}(W(K))$ by (ii).
For $x \in W$, we have 
$f^{N-1}(f(x))=f^N(x) \in f^N(W) \subset W$, so $f(x) \in (f^{N-1})^{-1}(W)$.
For $x \in f^{-i}(W)$ with $1 \leq i \leq N-1$, we have 
$f^{i-1}(f(x))=f^i(x) \in W$, so $f(x) \in (f^{i-1})^{-1}(W)$.
Thus $f(\bigcup_{i=0}^{N-1} (f^i)^{-1}(W)) \subset \bigcup_{i=0}^{N-1} (f^i)^{-1}(W)$.
Hence $\bigcup_{i=0}^{N-1} (f^i)^{-1}(W)$ is an $f$-invariant proper closed subset.
\end{proof}

We introduce the lexicographic order on the pairs $(\delta_f, l_f)$.

\begin{defn}\label{defn3.6}
For $(\delta_1, l_1), (\delta_2, l_2) \in \mathbb R_{\geq 1} \times \mathbb Z_{\geq 0}$,
$(\delta_1, l_1) \leq (\delta_2, l_2)$ if 
$\delta_1 < \delta_2$ holds, or $\delta_1=\delta_2$ and $l_1 \leq l_2$ hold.
\end{defn}

\begin{lem}\label{lem3.7}
Let $X, Y$ be smooth projective varieties and $f, g$ endomorphisms on $X, Y$, respectively.
\begin{itemize}
\item[(i)]
$(\delta_{f \times g}, l_{f \times g})=\max \{ (\delta_f, l_f), (\delta_g, l_g) \}$.

\item[(ii)]
$$
\overline h_{f \times g}(x, y) \asymp 
\begin{cases}
\overline h_f(x)\ \ \ \mathrm{if}\ (\delta_f, l_f) > (\delta_g, l_g), \\
\overline h_f(x) + \overline h_g(y)\ \ \ \mathrm{if}\ (\delta_f, l_f) = (\delta_g, l_g), \\
\overline h_g(y)\ \ \ \mathrm{if}\ (\delta_f, l_f) < (\delta_g, l_g).
\end{cases}
$$
The lower ample canonical height $\underline h_{f \times g}$ is similar.

\item[(iii)]
Let $K \subset \overline{\mathbb Q}$ be any subfield where $X,Y,f,g$ are defined.
Then 
$$
Z_{f \times g}(K)=
\begin{cases}
Z_f(K) \times Y(K) \ \ \ \mathrm{if}\ (\delta_f, l_f) > (\delta_g, l_g), \\
Z_f(K) \times Z_g(K)
\ \ \ \mathrm{if}\ (\delta_f, l_f) = (\delta_g, l_g), \\
X(K) \times Z_g(K)\ \ \ \mathrm{if}\ (\delta_f, l_f) < (\delta_g, l_g).
\end{cases}
$$
\end{itemize}
\end{lem}

\begin{proof}
We may assume that $(\delta_f, l_f) \geq (\delta_g, l_g)$ without loss of generality.
Then $\delta_{f \times g}=\max \{ \delta_f, \delta_g \} = \delta_f$ 
by the product formula (cf.~\cite{Tru15}).

Let $p: X \times Y \to X$, $q: X \times Y \to Y$ be the projections.
Since $h_X \circ p + h_Y \circ q$ is an ample height on $X \times Y$,
$h_{X \times Y} \asymp h_X \circ p + h_Y \circ q$.
Take any non-negative integer $l$.
\begin{align*}
\frac{h_{X \times Y}((f \times g)^n(x,y))}{\delta_{f \times g}^n n^l}
&= \frac{h_{X \times Y}(f^n(x), g^n(y))}{\delta_f^n n^l} \\
&\asymp \frac{h_X(f^n(x))}{\delta_f^n n^{l}} 
+ \left( \frac{\delta_g}{\delta_f} \right)^n  \frac{h_Y(g^n(y))}{\delta_g^n n^{l}}.
\end{align*}

If $\delta_g < \delta_f$, then 
$\lim_n \left( \frac{\delta_g}{\delta_f} \right)^n \frac{h_Y(g^n(y))}{\delta_g^n n^{l}} =0$ 
for any $l$,
so $l_{f \times g}=l_f$ and $\overline h_{f \times g}(x,y) \asymp \overline h_f(x)$.

Assume that $\delta_f=\delta_g$. Then 
$$ \frac{h_{X \times Y}((f \times g)^n(x,y))}{\delta_{f \times g}^n n^l} 
\asymp \frac{h_X(f^n(x))}{\delta_f^n n^{l}} 
+ \frac{h_Y(g^n(y))}{\delta_g^n n^{l}}.$$
So $l_{f \times g}= \max \{ l_f, l_g \}=l_f$ and 
$$\overline h_{f \times g}(x,y) \asymp
\begin{cases}
\overline h_f(x)\ \ \ \mathrm{if}\ l_f > l_g, \\
\overline h_f(x) + \overline h_g(y)\ \ \ \mathrm{if}\ l_f = l_g.
\end{cases}
$$
Thus (i) and (ii) hold.
(iii) follows from (ii).
\end{proof}

\begin{lem}\label{lem3.8}
Let $X, Y$ be smooth projective varieties, $f, g$ endomorphisms on $X, Y$, respectively,
and $\pi: X \to Y$ a surjective morphism such that $\pi \circ f = g \circ \pi$.
\begin{itemize}
\item[(i)]
$(\delta_f, l_f) \geq (\delta_g, l_g)$.

\item[(ii)]
Assume that $(\delta_f, l_f)=(\delta_g, l_g)$.
Then $\overline h_g \circ \pi \prec \overline h_f$ 
and $\underline h_g \circ \pi \prec \underline h_f$.
In particular, let $K \subset \overline{\mathbb Q}$ be any subfield where 
all concerned are defined, then 
$Z_f(K) \subset \pi^{-1}(Z_g(K))$.

\item[(iii)]
Assume that $\pi$ is finite.
Then $(\delta_f, l_f)=(\delta_g, l_g)$, $\overline h_g \circ \pi \asymp \overline h_f$ 
and $\underline h_g \circ \pi \asymp \underline h_f$.
In particular,
$Z_f(\overline{\mathbb Q}) = \pi^{-1}(Z_g(\overline{\mathbb Q}))$.
\end{itemize}
\end{lem}

\begin{proof}
Since $h_Y \circ \pi \prec h_X$ (cf.~Lemma \ref{lem2.1}), 
there is a positive constant $C$ such that $h_Y \circ \pi \leq Ch_X$.

(i) The product formula implies that $\delta_f \geq \delta_g$ (cf.~\cite{Tru15}).
If $\delta_f > \delta_g$, then $(\delta_f, l_f) > (\delta_g, l_g)$.
Assume that $\delta_f=\delta_g$.
Take any $y \in Y(\overline{\mathbb Q})$.
We can take $x \in \pi^{-1}(y)$. Then 
$$\frac{h_Y(g^n(y))}{\delta_g^n n^{l_f}} = \frac{h_Y(\pi f^n(x))}{\delta_f^n n^{l_f}}
\leq C \frac{h_X(f^n(x))}{\delta_f^n n^{l_f}}.$$
So $\left\{ \frac{h_Y(g^n(y))}{\delta_g^n n^{l_f}} \right\}_{n=0}^\infty$ 
is upper bounded and therefore $l_g \leq l_f$.

(ii) By assumption,
$$\frac{h_Y(g^n \pi(x))}{\delta_g^n n^{l_g}}
= \frac{h_Y(\pi f^n(x))}{\delta_f^n n^{l_f}}
\leq C \frac{h_X(f^n(x))}{\delta_f^n n^{l_f}}.$$
So $\overline h_g \circ \pi \leq C \overline h_f$ and 
$\underline h_g \circ \pi \leq C \underline h_f$.

(iii) Since $\pi$ is finite, we have $\delta_f=\delta_g$ and $h_Y \circ \pi \asymp h_X$.
So we can take a positive constant $C'$ such that $h_X \leq C' h_Y \circ \pi$.
Then 
$$\frac{h_X(f^n(x))}{\delta_f^n n^{l_g}} \leq C' \frac{h_Y(g^n \pi(x))}{\delta_g^n n^{l_g}}.$$
So $l_f \leq l_g$. Combining with (i), we obtain $l_f=l_g$.
By the above inequality, $\overline h_f \leq C' \overline h_g \circ \pi$ and 
$\underline h_f \leq C' \underline h_g \circ \pi$.
Combining with (ii), we obtain $\overline h_f \asymp \overline h_g \circ \pi$ and 
$\underline h_f \asymp \underline h_g \circ \pi$.
\end{proof}


The following is a version of the Chevalley--Weil theorem 
(see e.g.~\cite[4.2]{Ser97} and \cite[Exercise C.7]{HiSi00}).

\begin{thm}[Chevalley--Weil]\label{thm_CW}
Let $X,Y$ be normal projective varieties and $\phi: X \to Y$ an  
\'etale morphism 
which are defined over a number field $K$.
Then there is a finite extension $L$ of $K$ such that 
$\phi^{-1}(Y(K)) \subset X(L)$.
In particular, $X$ is potentially dense if and only if $Y$ is potentially dense.
\end{thm}

\begin{lem}\label{lem3.8.1}
Let $X,Y$ be smooth projective varieties, $f,g$ endomorphisms on $X,Y$ respectively,
and $\phi: X \to Y$ be a finite morphism such that $\phi \circ f=g \circ \phi$.
\begin{itemize}
\item[(i)]
Conjecture \ref{conj_zf} holds for $f$ if it holds for $g$.
\item[(ii)]
Assume that $\phi$ is \'etale.
Then Conjecture \ref{conj_zf} holds for $f$ if and only if it holds for $g$.
\end{itemize}
\end{lem}

\begin{proof}
Let $K$ be any number field where all concerned are defined.
Take any positive integer $d$.

(i) 
By assumption, there is a $g$-invariant proper closed subset $W \subset Y$ such that 
$Z_g(K) \subset W(K)$.
Then $Z_f(K) \subset \phi^{-1}(Z_g(K)) \subset \phi^{-1}(W(K))$ 
by Lemma \ref{lem3.8} (ii) and 
$\phi^{-1}(W)$ is an $f$-invariant proper closed subset of $X$.

(ii) Assume that Conjecture \ref{conj_zf} holds for $f$.
By Theorem \ref{thm_CW},
there is a finite extension $L$ of $K$ such that $\phi^{-1}(Y(K)) \subset X(L)$.
We can take an $f$-invariant proper closed subset $V \subset X$ satisfying 
$Z_f(L) \subset V$.
Take any $y \in Z_g(K)$.
Then we can take $x \in X(L)$ such that $\phi(x)=y$.
Lemma \ref{lem3.8} (iii) implies that 
$x \in Z_f(\overline{\mathbb Q}) \cap X(L)=Z_f(L) \subset V$.
So $y=\phi(x) \in \phi(V)$.
Thus $Z_g(K) \subset \phi(V)$.
Clearly $\phi(V)$ is a $g$-invariant proper closed subset of $Y$.
\end{proof}

For an endomorphism $f$ with $\delta_f>1$ and $l_f=0$,
the following holds.

\begin{thm}\label{thm3.11}
Let $X$ be a smooth projective variety and $f$ an endomorphism on $X$ with 
$\delta_f>1$ and $l_f=0$.
Then  
$X(\overline{\mathbb Q}) \setminus Z_f(\overline{\mathbb Q})$ is dense 
in $X$.
\end{thm}

\begin{proof}
The proof is almost same as the proof of \cite[Theorem 1.6]{MSS17}.

Using Theorem \ref{thm_pfb}, 
take a nef $\mathbb R$-divisor $D$ on $X$ such that $D \not\equiv 0$ and 
$f^*D \equiv \delta_fD$.
Since $h_D \prec h_X$, we can take $M_1 >0$ such that $h_D \leq M_1 h_X$.
Then 
$$\hat h_{D,f}= \lim_{n \to \infty} \frac{h_D \circ f^n}{\delta_f^n}
\leq \liminf_{n \to \infty} M_1 \frac{h_X \circ f^n}{\delta_f^n} 
= M_1 \underline h_f.$$
So $\underline h_f(x)>0$ if $\hat h_{D,f}(x)>0$.

Take any proper closed subset $V \subset X$ and a very ample divisor $H$ on $X$.
By \cite[Lemma 20]{KaSi16a}, it follows that $(D \cdot H^{\dim X-1})>0$.
Take $H_1,\ldots,H_{\dim X -1} \in |H|$ such that 
$C=H_1 \cap \cdots \cap H_{\dim X-1}$ is a smooth curve with $C \not\subset V$.
Since $\hat h_{D,f}(x)=h_D+O(\sqrt{h_X})$, we can take $M_2>0$ such that 
$\hat h_{D,f} \geq h_D - M_2 \sqrt{h_X}$.
Now $h_D|_C, h_X|_C$ are ample heights and $h_X|_C \geq 1$,
so we can take $M_3, M_4, M_5 >0$ such that 
$h_D|_C \geq M_3 h_C - M_4$, $h_X|_C \geq M_5 h_C$.
Then 
$$\hat h_{D,f}|_C \geq M_3 h_C - M_4 -M_2 \sqrt{M_5 h_C}
= \sqrt{h_C} (M_3 \sqrt{h_C} -M_2 \sqrt{M_5}) -M_4.$$
So, by the Northcott finiteness theorem, there are infinitely many points on $C$ at which 
$\hat h_{D,f}$ has positive value.
Take $x \in C \setminus V$ such that $\hat h_{D,f}(x)>0$.
Then $x \not\in V$ and $\underline h_f(x)>0$.
\end{proof}

Finally, we prove that $\delta_f$ and the pair $(\delta_f,l_f)$ are characterized 
in terms of 
ample canonical heights.

\begin{defn}\label{defn3.9}
Let $X$ be a smooth projective variety, $f$ an endomorphism on $X$, and 
$(\delta, l) \in \mathbb R_{\geq 1} \times \mathbb Z_{\geq 0}$.
We set 
$$\overline h_{f,\delta,l}(x) = \limsup_{n \to \infty} \frac{h_X(f^n(x))}{\delta^n n^l},
\underline h_{f,\delta,l}(x) = \liminf_{n \to \infty} \frac{h_X(f^n(x))}{\delta^n n^l}.$$
\end{defn}

\begin{thm}\label{thm3.10}
Let $X$ be a smooth projective variety and $f$ an endomorphism on $X$.
Then 
$$\delta_f= \min \{ \delta \in \mathbb R_{\geq 1}  \mid
\overline h_{f,\delta+\varepsilon,0}(x) < \infty \ \mathrm{for\ any\ }\varepsilon >0\ 
\mathrm{and\ } x \in X(\overline{\mathbb Q}) \}.$$
\end{thm}

\begin{proof}
By  Theorem \ref{thm_KSM}, $\overline h_{f, \delta_f+\varepsilon,0} < \infty$ for 
every $\varepsilon >0$.

Take any $\delta \in \mathbb R_{\geq 1}$ such that 
$\overline h_{f,\delta+\varepsilon,0} < \infty$ for every $\varepsilon >0$.
By Theorem \ref{thm_existence}, there is a point $x \in X(\overline{\mathbb Q})$ 
such that $\alpha_f(x)=\delta_f$.
Take any $\varepsilon >0$.
Then there is a positive constant $C$ such that $h_X(f^n(x)) \leq C (\delta+\varepsilon)^n$ 
for all $n$ since $\overline h_{f,\delta+\varepsilon,0}(x)<\infty$. So 
$$\delta_f =\alpha_f(x)= \lim_{n \to \infty} h_X(f^n(x))^{1/n} 
\leq \lim_{n \to \infty} C^{1/n}(\delta+\varepsilon) =\delta+\varepsilon.$$
Since $\varepsilon$ is arbitrary, $\delta_f \leq \delta$.
\end{proof}

\begin{thm}\label{thm3.12}
Let $X$ be a smooth projective variety and $f$ an endomorphism on $X$.
Then
$$(\delta_f,l_f)= \min \{ (\delta,l) \in \mathbb R_{\geq 1} \times \mathbb Z_{\geq 0} \mid
\overline h_{f,\delta,l}(x) < \infty \ \mathrm{for\ any\ }
x \in X(\overline{\mathbb Q}) \}.$$
\end{thm}

\begin{proof}
By definition, $\overline h_{f,\delta_f,l_f}= \overline h_f < \infty$.

Take any $(\delta,l) \in \mathbb R_{\geq 1} \times \mathbb Z_{\geq 0}$ such that 
$\overline h_{f,\delta,l} < \infty$.
Then $\overline h_{f, \delta+\varepsilon,0} < \infty$ for any $\varepsilon >0$.
So $\delta_f \leq \delta$ by Theorem \ref{thm3.10}.
If $\delta_f < \delta$, then $(\delta_f,l_f) < (\delta,l)$.
If $\delta_f=\delta$, then $l_f \leq l$ by definition of $l_f$.
Eventually, $(\delta_f,l_f) \leq (\delta,l)$.
\end{proof}


\section{Varieties with small Picard numbers}\label{sec4}

This section treats ample canonical heights for endomorphisms 
on smooth projective varieties with small Picard numbers.

If $X$ is a smooth projective variety with $\rho(X)=1$ and $f$ is an endomorphism with 
$\delta_f>1$, we can take an ample divisor $H$ such that $f^*H \equiv \delta_f H$.
So the following includes the case when $\rho(X)=1$.

\begin{thm}\label{thm4.1}
Let $X$ be a smooth projective variety and $f$ an endomorphism on $X$ with 
$\delta_f >1$.
Assume that there is an ample $\mathbb R$-divisor $H$ on $X$ such that 
$f^*H \equiv \delta_f H$.
\begin{itemize}
\item[(i)]
We have $l_f=0$ and $\overline h_f \asymp \underline h_f \asymp \hat h_{H,f}$.
\item[(ii)]
We have $Z_f(\overline{\mathbb Q})=\Preper_f(\overline{\mathbb Q})$,
and $\Preper_f(K)$ is finite for any number field $K$.
\end{itemize}
\end{thm}

\begin{proof}
(i) Take $h_H$ as satisfying $h_H \geq 1$.
Then $h_X \asymp h_H$, so $l_f=0$ and 
$\overline h_f, \underline h_f \asymp \hat h_{H,f}$.

(ii) 
By Proposition \ref{prop3.2} (ii), $Z_f(\overline{\mathbb Q})$ contains all $f$-preperiodic 
points.
Conversely, take any non-$f$-preperiodic point $x$.
Since $\hat h_{H,f} = h_H + O(\sqrt{h_H})$, we can take $C>0$ such that 
$$\hat h_{H,f} \geq h_H - C\sqrt{h_H} = \sqrt{h_H}(\sqrt{h_H} - C).$$
Therefore $\hat h_{H,f}(f^k(x)) >0$ for some $k$ by the Northcott finiteness theorem.
Then $\hat h_{H,f}(x)>0$ since $\hat h_{H,f}(f^k(x))=\delta_f^k \hat h_{H,f}(x)$.
So $\underline h_f(x)>0$ 
because $\underline h_f \asymp \hat h_{H,f}$.
Thus $Z_f(\overline{\mathbb Q})=\Preper_f(\overline{\mathbb Q})$.

Let $K$ be any number field.
The inequality $\hat h_{H,f} \geq \sqrt{h_H}(\sqrt{h_H} - C)$ 
implies that $\sqrt{h_H(x)}(\sqrt{h_H(x)} - C) \leq 0$ for any 
$x \in Z_f(\overline{\mathbb Q})$.
So the Northcott finiteness theorem implies that 
$\Preper_f(K)=Z_f(K)$  is finite.
\end{proof}

Next, we consider the case when $\rho(X) \leq 2$.

\begin{thm}\label{thm4.2}
Let $X$ be a smooth projective variety with $\rho(X) \leq 2$ and $f$ an endomorphism 
on $X$ with $\delta_f >1$.
\begin{itemize}
\item[(i)]
We have $l_f=0$. 
\item[(ii)]
Assume that $f$ is an automorphism.
Then there is a nef canonical height $\hat h_{D,f}$ such that 
$\overline h_f \asymp \underline h_f \asymp \hat h_{D,f}$.
Moreover, 
we have $Z_f(\overline{\mathbb Q})=\Per_f(\overline{\mathbb Q})$, 
and $\Per_f(K)$ is finite for any number field $K$.
\end{itemize}
\end{thm}

To prove Theorem \ref{thm4.2}, we use the following lemma.

\begin{lem}\label{lem4.3}
Let $X$ be a smooth projective variety and $f$ an endomorphism on $X$ with $\delta_f>1$.
\begin{itemize}
\item[(i)]
We have 
$$\lim_{n \to \infty} \frac{\sqrt{h_X(f^n(x))}}{\delta_f^n}=0$$
for every $x \in X(\overline{\mathbb Q})$.
\item[(ii)]
Let $D$ be an $\mathbb R$-divisor on $X$ such that $f^*D \equiv \lambda D$ with 
$0<\lambda < \delta_f$.
Then 
$$\lim_{n \to \infty} \frac{h_D(f^n(x))}{\delta_f^n}=0$$
for every $x \in X(\overline{\mathbb Q})$.
\end{itemize}
\end{lem}

\begin{proof}
(i) Take $\varepsilon >0$ such that $\delta_f+\varepsilon < \delta_f^2$.
By Theorem \ref{thm_KSM}, 
there is a positive constant $C$ 
such that 
$h_X \circ f^n \leq C (\delta_f+\varepsilon)^n h_X$ for all $n$.
Then 
$$\frac{\sqrt{h_X \circ f^n}}{\delta_f^n} 
\leq \frac{\sqrt{C (\delta_f+\varepsilon)^n h_X}}{\delta_f^n}
= \sqrt{C} \left( \frac{\delta_f+\varepsilon}{\delta_f^2} \right)^{n/2} \sqrt{h_X}.$$
So the assertion follows.

(ii) Set $\phi=h_D \circ f - \lambda h_D$. Since $\phi=O(\sqrt{h_X})$,
there is a positive constant $C'$ such that $\phi \leq C' \sqrt{h_X}$.
Then 
\begin{align*}
h_D \circ f^n
&= \sum_{k=1}^n \lambda^{n-k} (h_D \circ f^k - \lambda h_D \circ f^{k-1}) + \lambda^n h_D \\
&= \sum_{k=1}^n \lambda^{n-k} \phi \circ f^{k-1} + \lambda^n h_D.
\end{align*}
So 
\begin{align*}
| h_D \circ f^n | 
&\leq \sum_{k=1}^n \lambda^{n-k} |\phi \circ f^{k-1}| + \lambda^n |h_D| \\
&\leq \sum_{k=1}^n \lambda^{n-k} C' \sqrt{h_X \circ f^{k-1}} + \lambda^n |h_D| \\
&\leq \sum_{k=1}^n \lambda^{n-k} C' 
\sqrt{C(\delta_f+\varepsilon)^{k-1}h_X} + \lambda^n |h_D| \\
&\leq C' \sqrt{C}  
\sum_{k=1}^n \lambda^{n-k}  (\delta_f+\varepsilon)^{(k-1)/2} \sqrt{h_X} + \lambda^n |h_D| \\
&\leq C' \sqrt{C} n \mu^{n-1}  \sqrt{h_X} + \lambda^n |h_D|,
\end{align*}
where $\mu=\max \{ \lambda, \sqrt{\delta_f+\varepsilon} \} <\delta_f$.
So $\lim_n h_D(f^n(x))/\delta_f^n =0$ for every $x$.
\end{proof}

\begin{proof}[Proof of Theorem \ref{thm4.2}]
If $\rho(X)=1$, the assertion follows from Theorem \ref{thm4.1}.
So we may assume $\rho(X)=2$.
By the Perron--Frobenius--Birkhoff theorem (Theorem \ref{thm_pfb}),
we have a nef $\mathbb R$-divisor $D \not\equiv 0$ such that $f^*D \equiv \delta_f D$.
If $D$ is ample, then the proof is reduced to Theorem \ref{thm4.1}.
So we may assume that $D$ is not ample.
Then the numerical class of $D$ is on one of two edges of the nef cone of $X$.
Take a nef $\mathbb R$-divisor $D' \not\equiv 0$ whose numerical class is on the other 
edge of the nef cone.
Then $f^*D' \equiv \lambda D'$ for some $0< \lambda \leq \delta_f$ 
since $f^*$ is an automorphism which preserves the boundary of the nef cone.
Since $A=D+D'$ is ample, taking $h_{D}, h_{D'}, h_A$ as satisfying 
$h_A=h_{D}+h_{D'} \geq 1$,
we have $h_X \asymp h_A$.

(i) Since $h_A = \hat h_{D,f}+ h_{D'}+O(\sqrt{h_A})$,
we can take $C>0$ such that $h_A \leq \hat h_{D, f}+ h_{D'} + C\sqrt{h_A}$.
Then 
$$\frac{h_A \circ f^n}{\delta_f^n} \leq \hat h_{D, f} + \frac{h_{D'} \circ f^n}{\delta_f^n}
+ C\frac{\sqrt{h_A \circ f^n}}{\delta_f^n}.$$
Take any point $x \in X(\overline{\mathbb Q})$.
If $\lambda < \delta_f$, then $\left\{ \frac{h_{D'}(f^n(x))}{\delta_f^n} \right\}_n$ converges 
to 0 by Lemma \ref{lem4.3} (ii).
If $\lambda = \delta_f$, then $\left\{ \frac{h_{D'}(f^n(x))}{\delta_f^n} \right\}_n$ converges 
to $\hat h_{D',f}(x)$.
So $\left\{ \frac{h_{D'}(f^n(x))}{\delta_f^n} \right\}_n$ converges in any case.
Furthermore, 
$\left\{ \frac{\sqrt{h_A(f^n(x))}}{\delta_f^n} \right\}_n$ converges to 0 
by Lemma \ref{lem4.3} (i).
Therefore $\left\{ \frac{h_A(f^n(x))}{\delta_f^n} \right\}_n$ is upper bounded.
So it follows that $l_f=0$.

(ii) Since $f$ is an automorphism, the inverse of $f^*: N^1(X) \to N^1(X)$ is  
$(f^{-1})^*: N^1(X) \to N^1(X)$. 
So $\delta_f \lambda = |\det(f^*)|=1$ and therefore $\lambda= \delta_f^{-1}$.
Now we have $h_A = \hat h_{D,f}+ \hat h_{D',f^{-1}}+O(\sqrt{h_A})$.
Set $\phi= h_A - \hat h_{D,f}- \hat h_{D',f^{-1}}$.
Then 
$$ \frac{h_A \circ f^n}{\delta_f^n} = \hat h_{D,f}+ \frac{\hat h_{D',f^{-1}}}{\delta_f^{2n}}
+\frac{\phi \circ f^n}{\delta_f^n}.$$
So $\lim_n \delta_f^{-n} h_A \circ f^n = \hat h_{D,f}$ 
by Lemma \ref{lem4.3}.
Since $h_A \asymp h_X$, we have $\overline h_f, \underline h_f \asymp \hat h_{D,f}$.

We can take $C'>0$ such that 
$\hat h_{D,f}+ \hat h_{D',f^{-1}} \geq h_A - C'\sqrt{h_A} = \sqrt{h_A}(\sqrt{h_A}-C')$.
Take any non-$f$-preperiodic point $x$.
Then 
$$\{ \hat h_{D,f}(f^n(x))+ \hat h_{D',f^{-1}}(f^n(x)) \}_{n=0}^\infty
= \{ \delta_f^n \hat h_{D,f}(x)+ \delta_f^{-n} \hat h_{D',f^{-1}}(x) \}_{n=0}^\infty$$ 
is not upper bounded 
by the Northcott finiteness theorem.
So $\hat h_{D,f}(x)$ must be positive.
Since $\underline h_f \asymp \hat h_{D,f}$, we obtain $\underline h_f(x)>0$.
Therefore $Z_f(\overline{\mathbb Q})=\Preper_f(\overline{\mathbb Q})
=\Per_f(\overline{\mathbb Q})$.

By the same argument for $f^{-1}$,
we obtain $Z_{f^{-1}}(\overline{\mathbb Q})=\Per_{f^{-1}}(\overline{\mathbb Q})$.
Clearly $\Per_{f}(\overline{\mathbb Q})=\Per_{f^{-1}}(\overline{\mathbb Q})$,
so we have $Z_{f^{-1}}(\overline{\mathbb Q})=Z_{f}(\overline{\mathbb Q})$.

Take any $x \in Z_f(\overline{\mathbb Q})$.
Then $\hat h_{D,f}(x)=0$ since $\hat h_{D,f} \asymp \underline h_f$.
Moreover, since $x \in Z_{f^{-1}}(\overline{\mathbb Q})$ 
and $\hat h_{D',f^{-1}} \asymp \underline h_{f^{-1}}$,
we have $\hat h_{D',f^{-1}}(x)=0$.
Then the inequality 
$\hat h_{D,f}+ \hat h_{D',f^{-1}} \geq \sqrt{h_A}(\sqrt{h_A}-C')$ 
implies that $\sqrt{h_A(x)}(\sqrt{h_A(x)}-C') \leq 0$ 
for $x \in Z_f(\overline{\mathbb Q})=\Per_f(\overline{\mathbb Q})$.
So the Northcott finiteness theorem deduces that 
$Z_f(K)=\Per_f(K)$ is finite 
for any number field $K$.
\end{proof}

At last, we consider a Calabi--Yau threefold with Picard number $\leq 3$ 
and an automorphism on it.
The arguments here is based on \cite{LOP17} and \cite{LOP16}.
To obtain a result, we need the following conjecture.

\begin{conj}[The abundance conjecture for Ricci flat manifolds, 
{\cite[Conjecture 4.8]{LOP16}}]\label{conj4.4}
Let $X$ be a smooth projective variety with $K_X \sim 0$ and $H^1(X,\mathcal O_X)=0$.
Then any nef Cartier divisor on $X$ is semiample.
\end{conj}

For an automorphism on a Calabi--Yau threefold with Picard number $\leq 3$,
a precise description of $Z_f$ is not obtained at the moment,
but we can prove Conjecture \ref{conj_main} if we assume Conjecture \ref{conj4.4}.

\begin{thm}\label{thm4.5}
Let $X$ be a Calabi--Yau threefold (i.e.~a projective threefold with 
$K_X \sim 0$ and 
$\pi_1(X)=0$) with $\rho(X) \leq 3$ and $f$ an automorphism on $X$ 
with $\delta_f>1$.
\begin{itemize}
\item[(i)] 
We have $l_f=0$, and 
there is a nef canonical height $\hat h_{D^+,f}$ such that 
$\overline h_f \asymp \underline h_f \asymp \hat h_{D^+,f}$.
\item[(ii)]
Assume Conjecture \ref{conj4.4}.
Then Conjecture \ref{conj_main} holds for $f$.
\end{itemize}
\end{thm}

\begin{rem}\label{rem4.6}
As we will see in the proof, $\rho(X)$ is automatically equal to 3 under the assumption of 
Theorem \ref{thm4.5}.
However the author does not know 
any example of $(X,f)$ in the theorem 
at the moment.
\end{rem}

\begin{proof}[Proof of Theorem \ref{thm4.5}]
Since $f$ is an automorphism, $f^*: N^1(X) \to N^1(X)$ has the inverse 
$(f^{-1})^*: N^1(X) \to N^1(X)$ and so $|\det(f^*)|=1$.
Now we have $\rho(f^*)=\delta_f>1$,
so $f^*$ has an eigenvalue whose absolute value is in $(0,1)$.
Hence $\delta_{f^{-1}} = \rho((f^{-1})^*)=\rho((f^*)^{-1}) >1$.
We can take nef $\mathbb R$-divisors $D^+, D^-$ on $X$ such that 
$D^+,D^- \not\sim_{\mathbb R} 0$ and 
$f^*D^+ \sim_{\mathbb R} \delta_f D^+$, $(f^{-1})^*D^- \sim_{\mathbb R} \delta_{f^{-1}} D^-$.
Note that we have these in $\mathbb R$-linear equivalence,
not numerical equivalence, since $q(X)=0$.
Since $f$ is of infinite order as an element of $\Aut(X)$,
$\Aut(X)$ is an infinite group.
Then it follows that $c_2=c_2(X) \neq 0$ in $N_1(X)_{\mathbb R}$ 
(cf.~\cite{LOP17}).
Here $f_*c_2 = c_2$ and $f^*$ is the adjoint of $f_*$, so there is an $\mathbb R$-divisor 
$D_0$ such that $D_0 \not\sim_{\mathbb R} 0$ and $f^*D_0 \sim_{\mathbb R} D_0$.

Eventually, we have three eigenvectors $D^+,D^-,D_0$ of $f^*$ 
with three different eigenvalues 
$\delta_f, \delta_{f^{-1}}^{-1}, 1$, respectively.
Then $\rho(X)=3$ since $D^+,D^-,D_0$ are linearly independent.
Since $|\det(f^*)|=1$, we have $\delta_{f^{-1}}=\delta_f$.
Now $D^+, D^-$ are in the nef cone of $X$.
So, replacing $D_0$ by a non-zero multiple,
we may assume that $D=D^+ + D^- + D_0$ is ample.
Taking $h_{D^+}, h_{D^-}, h_{D_0}$ and  $h_D$ as satisfying 
$h_D=h_{D^+}+h_{D^-} + h_{D_0} \geq 1$,
we have $h_X \asymp h_D$.
Moreover, $h_D = \hat h_{D^+,f}+ \hat h_{D^-,f^{-1}} + h_{D_0} +O(1)$.

(i) Set $\phi= h_D - \hat h_{D^+,f}- \hat h_{D^-,f^{-1}} - h_{D_0}$.
Then 
$$ \frac{h_D \circ f^n}{\delta_f^n} 
= \hat h_{D^+,f}+ \frac{\hat h_{D^-,f^{-1}}}{\delta_f^{2n}} + \frac{h_{D_0} \circ f^n}{\delta_f^n} 
+\frac{\phi \circ f^n}{\delta_f^n}.$$
So $\lim_n \delta_f^{-n}  h_D \circ f^n= \hat h_{D^+,f}$ 
by Lemma \ref{lem4.3}.
Since $h_D \asymp h_X$, we have $l_f=0$ and 
$\overline h_f, \underline h_f \asymp \hat h_{D^+,f}$.

(ii) Since $(D^+ \cdot c_2)=(D^+ \cdot f_* c_2)=(f^*D^+ \cdot c_2)=\delta_f(D^+ \cdot c_2)$, 
we have $(D^+ \cdot c_2)=0$.
Similarly $(D^- \cdot c_2)=0$.
Let $c_2^{\perp} \subset N^1(X)_{\mathbb R}$ be the subspace of $N^1(X)_{\mathbb R}$ 
consisting of the elements whose intersection with $c_2$ is zero.
Then $c_2^{\perp}$ is a 2-dimensional rational subspace generated by $D^+,D^-$.
So $\mathbb R_{>0} D^+ + \mathbb R_{>0} D^-$ contains rational points.
Therefore $B=aD^+ + bD^-$ is a nef Cartier divisor for some $a,b >0$.
Applying Conjecture \ref{conj4.4} to $B$, it follows that $B$ is semiample.

Take any dense $f$-orbit $O_f(x)$.
By Corollary \ref{cor_Northcott}, $\{ h_B(f^n(x)) \}_{n=0}^\infty$ is upper unbounded.
Since $h_B=ah_{D^+}+bh_{D^-}+O(1)=a \hat h_{D^+,f}+b \hat h_{D^-,f^{-1}}+O(1)$,
$$\{ a\hat h_{D^+,f}(f^n(x))+b\hat h_{D^-,f^{-1}}(f^n(x)) \}_{n=0}^\infty 
= \{ a\delta_f^n \hat h_{D^+,f}(x)+b\delta_f^{-n} \hat h_{D^-,f^{-1}}(x) \}_{n=0}^\infty$$ is also 
upper unbounded.
Therefore $\hat h_{D^+,f}(x)$ must be positive.
Since $\underline h_f \asymp \hat h_{D^+,f}$, 
we obtain $\underline h_f(x)>0$.
\end{proof}


\section{Abelian varieties}\label{sec5}

For an abelian group $G$,
$G_{\tors}$ denotes the set of torsion elements of $G$.
The main result in this section is the following.

\begin{thm}\label{thm5.1}
Let $X$ be an abelian variety and $f$ an endomorphism 
(which is not necessarily an isogeny) 
on $X$ with $\delta_f>1$.
Then there is 
a proper abelian subvariety $B \subset X$ and a point $P_0 \in X(\overline{\mathbb Q})$ 
such that $B+P_0$ is $f$-invariant and 
$Z_f(\overline{\mathbb Q})=B(\overline{\mathbb Q})+P_0+X(\overline{\mathbb Q})_{\tors}$.
Moreover, Conjecture \ref{conj_zf} holds for $f$.
\end{thm}

\begin{proof}
\begin{step}
First we assume that $f \in \End(X)$.
It is well-known that 
$\End(X)_{\mathbb Q}$ is a finite-dimensional $\mathbb Q$-vector space 
(cf.~\cite[Chapter IV, Section 19, Theorem 3]{Mum70}).
So the subspace 
generated by $\id_X, f, f^2,\ldots$ also has finite dimension.
Hence we can take a positive integer $m$ and $c_1,c_2,\ldots,c_m \in \mathbb Q$ such that 
$f^m=c_1f^{m-1}+c_2 f^{m-2}+ \cdots +c_m$ in $\End(X)_{\mathbb Q}$.
Let $A=(a_{i,j})_{i,j}$ be an $m \times m$-matrix defined by 
$a_{i,i+1}=1$ for $1 \leq i \leq m-1$, $a_{m,j}=c_{m-j+1}$ for $1 \leq j \leq m$, and 
$a_{i,j}=0$ otherwise.
Then we have 
$$ A
\begin{pmatrix}
f^k \\ f^{k+1}  \\ \vdots \\ f^{k+m-1}
\end{pmatrix}
= 
\begin{pmatrix}
f^{k+1} \\ f^{k+2}  \\ \vdots \\ f^{k+m}
\end{pmatrix}
$$
for every $k \in \mathbb Z_{\geq 0}$.
So, setting $\vec{e}_1=(1,0,0,\ldots,0)$ and 
$\vec{f}=
\begin{pmatrix}
\id_X \\ f \\ f^2 \\ \vdots \\ f^{m-1}
\end{pmatrix}
$,
we have $f^n=\vec{e}_1A^n\vec{f}$.
Take a complex invertible $m \times m$-matrix $P$ such that 
$\Lambda=P^{-1}AP$ is a Jordan normal form of $A$.
Then $f^n=\vec{e}_1 P \Lambda^n P^{-1}\vec{f}$.
Therefore $f^n$ is represented as 
$f^n=\sum_{i=1}^N \lambda_i^n n^{l_i} g_i$,
where $\lambda_i \in \mathbb C$, $l_i \in \mathbb Z_{\geq 0}$, 
and $g_i \in (\sum_{i=0}^{m-1} \mathbb C f^i) \setminus \{0\}$ with 
$(|\lambda_1|, l_1) > (|\lambda_2|, l_2) > \cdots > (|\lambda_N|,l_N)$ with respect to 
the lexicographic order.

Let $\hat h_X = \langle \cdot,\cdot \rangle$ be a N\'eron--Tate height on $X$.
Set $M=X(\overline{\mathbb Q})/X(\overline{\mathbb Q})_{\tors}$.
Then $\langle \cdot,\cdot \rangle$ is reduced to a $\mathbb Z$-bilinear form 
on $M \times M$.
Set $V_{\mathbb K}=M \otimes \mathbb K$ 
($\mathbb K=\mathbb Q, \mathbb R$ or $\mathbb C$).
Then  $\langle \cdot,\cdot \rangle$ (and so $\hat h_X$) is extended 
to a positive definite hermitian form on $V_{\mathbb C} \times V_{\mathbb C}$ by  
$\langle x,\alpha y\rangle = \alpha \langle x,y \rangle 
=\langle \overline{\alpha} x,y \rangle$ for $x,y \in M$ and $\alpha \in \mathbb C$ 
(cf.~\cite[Proposition B.5.3]{HiSi00}).
Take any $x \in V_{\mathbb R}$.
Then 
\begin{align*}
\hat h_X(f^n(x))
&= \hat h_X \left( \sum_{i=1}^N \lambda_i^n n^{l_i} g_i(x) \right) \\
&= |\lambda_1|^{2n} n^{2l_1} \hat h_X(g_1(x)) + o \left( |\lambda_1|^{2n} n^{2l_1} \right) 
\ \ (n \to \infty).
\end{align*}
Write $g_1=\phi+\sqrt{-1}\psi$ with $\phi, \psi \in \End(X)_{\mathbb R}$.
Then $g_1(x)=\phi(x)+\sqrt{-1}\psi(x)$ with $\phi(x),\psi(x) \in V_{\mathbb R}$.
So,  for any $x \in V_{\mathbb R}$, $g_1(x)=0$ if and only if $\phi(x)=\psi(x)=0$.
We use the following lemma due to Kawaguchi and Silverman.

\begin{lem}[{\cite[Lemma 30]{KaSi16b}}]\label{lem5.4}
Let $V, W$ be $\mathbb Q$-vector spaces,
$\mathcal D \subset \Hom_{\mathbb Q}(V,W)$ a $\mathbb Q$-vector subspace,
and $\alpha \in \mathcal D_{\mathbb R}$.
Then there are some $\beta_1,\ldots,\beta_m \in \mathcal D$ such that 
$\alpha(v)=0$ if and only if $\beta_1(v)=\cdots=\beta_m(v)=0$ for any $v \in V$.
\end{lem}

By this lemma, we can take $\beta_1,\ldots,\beta_k \in \End(X)_{\mathbb Q}$ such that 
$\phi(x)=0$ if and only if $\beta_1(x)=\cdots=\beta_k(x)=0$ for any $x \in V_{\mathbb Q}$.
Replacing $\beta_i$ by a multiple, we may assume that $\beta_i \in \End(X)$.
Similarly we can take $\gamma_1,\ldots,\gamma_l \in \End(X)$ such that 
$\psi(x)=0$ if and only if $\gamma_1(x)=\cdots=\gamma_l(x)=0$ 
for any $x \in V_{\mathbb Q}$.
Each member of $\End(X)$ has a kernel as an algebraic subgroup of $X$,
so there is an abelian subvariety $B \subset X$ such that 
$\{ x \in X(\overline{\mathbb Q}) \mid g_1(x)=0 \ \mathrm{in} \ V_{\mathbb C} \}
=B(\overline{\mathbb Q}) + X(\overline{\mathbb Q})_{\tors}$.
Here $B$ is a proper abelian subvariety since $g_1 \neq 0$.

Using Theorem \ref{thm3.12}, we obtain $(|\lambda_1|^2,2l_1)=(\delta_f,l_f)$.
Eventually we have $\overline h_f(x) \asymp \underline h_f(x) \asymp \hat h_X(g_1(x))$ 
for every $x \in X(\overline{\mathbb Q})$.
Hence 
$Z_f(\overline{\mathbb Q})
=\{ x \in X(\overline{\mathbb Q}) \mid g_1(x)=0 \ \mathrm{in} \ V_{\mathbb C} \}
=B(\overline{\mathbb Q}) + X(\overline{\mathbb Q})_{\tors}$.
Since $Z_f(\overline{\mathbb Q})$ is $f$-invariant, $B$ is also $f$-invariant.
\end{step}

\begin{step}
Let us consider the general case.
Set $f=\tau_P \circ \phi$, where $\tau_P$ is the translation by $P$ and $\phi \in \End(X)$.

We use the following lemma due to Silverman 
(cf.~{\cite[Proof of Theorem 2]{Sil17}}).

\begin{lem}\label{lem5.9}
Let $X$ be an abelian variety and $\phi \in \End(X)$.
Then there are abelian subvarieties $X_1,X_2 \subset X$ such that 
\begin{itemize}
\item The addition 
$\lambda: X_1 \times X_2 \to X$, $\lambda(x_1,x_2)=x_1+x_2$ is an isogeny.

\item $\phi(X_i) \subset X_i$ $(i=1,2)$. Set $\phi_i=\phi|_{X_i}$.

\item $(\id_{X_1}-\phi_1)(X_1)=X_1$.

\item $\delta_{\phi_2}=1$.
\end{itemize}
\end{lem}
Take $P_i \in X_i$ such that $P=P_1+P_2$ and set $f_i=\tau_{P_i} \circ \phi_i$ $(i=1,2)$.
Then 
$\lambda(f_1(x_1),f_2(x_2))
=\phi_1(x_1)+P_1+\phi_2(x_2)+P_2
=\phi(x_1+x_2)+P=f(\lambda(x_1,x_2))$.
Thus $\lambda \circ (f_1 \times f_2) = f \circ \lambda$.
Since translation maps induce the identity map on $N^1(X)_{\mathbb R}$, we have 
$\delta_\phi=\delta_f>1$, $\delta_{\phi_1}=\delta_{f_1}$ and $\delta_{\phi_2}=\delta_{f_2}=1$.
So $\delta_{f_1}>1=\delta_{f_2}$.
By Lemma \ref{lem3.7} and Lemma \ref{lem3.8}, we have 
$\underline h_f \circ \lambda \asymp \underline h_{f_1 \times f_2} 
\asymp \underline h_{f_1}$ and so $Z_f=\lambda(Z_{f_1} \times X_2)$.

Since $(\id_{X_1}-\phi_1)(X_1)=X_1$, there is a point $P_0 \in X_1(\overline{\mathbb Q})$ 
such that $P_0-\phi_1(P_0)=P_1$.
Then 
$(f_1 \circ \tau_{P_0})(x_1)
=\phi_1(x_1+P_0)+P_1
=\phi_1(x_1)+\phi_1(P_0)+P_1
=\phi_1(x_1)+P_0
=(\tau_{P_0} \circ \phi_1)(x_1)$.
Thus $f_1 \circ \tau_{P_0}=\tau_{P_0} \circ \phi_1$.
By Lemma \ref{lem3.8} (iii), $\underline h_{f_1} \circ \tau_{P_0}=\underline h_{\phi_1}$ 
and so $Z_{f_1}=\tau_{P_0}(Z_{\phi_1})=Z_{\phi_1}+P_0$.
By Step 1, there is a proper abelian subvariety $B_1$ of $X_1$ such that 
$Z_{\phi_1}=B_1+(X_1)_{\tor}$.
As a consequence, we have 
\begin{align*}
Z_f
&=Z_{f_1}+X_2 \\
&=Z_{\phi_1}+P_0+X_2 \\
&=B_1+(X_1)_{\tor}+P_0+X_2 \\
&=(B_1+X_2)+P_0+(X_1)_{\tor}+(X_2)_{\tor} \\
&=B+P_0+X_{\tor},
\end{align*}
where we set $B=B_1+X_2$.
Then $B$ is a proper abelian subvariety of $X$.
We compute 
\begin{align*}
f(B+P_0)
&=\phi(B_1+X_2+P_0)+P \\
&= \phi_1(B_1)+\phi_2(X_2)+\phi_1(P_0)+P \\
&\subset B_1+X_2+\phi_1(P_0)+P \\
&= B_1+(X_2+P_2)+(\phi_1(P_0)+P_1) \\
&=B_1+X_2+P_0 \\
&=B+P_0.
\end{align*}
Thus $B+P_0$ is $f$-invariant.
\end{step}

\begin{step}
Finally we prove that Conjecture \ref{conj_zf} holds for $f$.
By Step 2, we have 
$Z_f(\overline{\mathbb Q})
= B(\overline{\mathbb Q}) +X(\overline{\mathbb Q})_{\tors}+P_0$
for a proper abelian subvariety $B \subset X$ and a point $P_0 \in X(\overline{\mathbb Q})$ 
such that $f(B+P_0) \subset B+P_0$.

Let $\pi: X \to Y=X/B$ be the quotient map.
Take a number field $K$ where 
$X,Y,\pi$ are defined and $P_0 \in X(K)$.
Take any $x \in Z_f(K)$.
Then $x-P_0 \in B(\overline{\mathbb Q})+X(\overline{\mathbb Q})_{\tor}$,
so $\pi(x-P_0) \in Y(K)_{\tor}$.
Let $N$ be the order of the finite group $Y(K)_{\tor}$,
then $\pi(N(x-P_0))=N\pi(x-P_0)=0$ and so $N(x-P_0) \in B(K)$.
Therefore $Z_f(K) \subset [N]^{-1}(B(K))+P_0$.
The Chevalley--Weil theorem (Theorem \ref{thm_CW}) implies that there is a finite extension 
$L \supset K$ such that $[N]^{-1}(B(K)) \subset B(L)$.
So we have $Z_f(K) \subset B(L)+P_0$.
\end{step}
\end{proof}

If $f$ is a self-isogeny on a power of an elliptic curve, we can compute 
$\delta_f$ and $l_f$ from the matrix representation of $f$.

\begin{defn}\label{defn5.1}
Take $A \in M_r(\mathbb C)$, a complex $r \times r$-matrix.
Let $\Lambda=\Lambda_1 \oplus \cdots \oplus \Lambda_t$ be the Jordan normal form 
of $A$, where $\Lambda_i$ is a Jordan block of size $(l_i+1) \times (l_i+1)$ with 
eigenvalue $\lambda_i$.
Then we define $l(A)$ as $l(A)=\max \{ l_i \mid |\lambda_i|=\rho(A) \}$.
\end{defn}

\begin{thm}\label{thm5.2}
Let $E$ be an elliptic curve, $X=E^r$, and $f \in \End(X)$ a self-isogeny.
Represent $f$ as $f(x_1, \ldots,x_r)=(\sum_j a_{1j}x_j,\ldots,\sum_j a_{rj}x_j)$, 
where $A=(a_{ij}) \in M_r(\End(E))$.
Then we have $\delta_f=\rho(A)^2$ and $l_f=2l(A)$.
\end{thm}

\begin{proof}
It is well-known that 
$\End(E)_{\mathbb Q}=\mathbb Q(\sqrt{-d})$ for some 
$d \in \mathbb Z_{\geq 0}$.
Set $\omega=\sqrt{-d}$.
Let $\hat h_E(x)=\langle x,x \rangle_E$ be a N\'eron--Tate height on $E$.
For $x=(x_1,\ldots,x_r), y=(y_1,\ldots,y_r) \in X(\overline{\mathbb Q})$,
set $\langle x,y \rangle_X=\sum_{i=1}^r \langle x_i,y_i \rangle_E$,
$\hat h_X(x)=\langle x,x \rangle_X=\sum_{i=1}^r \hat h_E(x_i)$.
Clearly $\hat h_X$ is a N\'eron--Tate height on $X$.
Set $M=E(\overline{\mathbb Q})/E(\overline{\mathbb Q})_{\tors}$.
Then $\langle \cdot,\cdot \rangle_E$ is reduced to  a $\mathbb Z$-bilinear form 
on $M \times M$.
Set $V_{\mathbb K}=M \otimes \mathbb K$ $(\mathbb K=\mathbb Q,\  \mathbb R 
\mathrm{\ or\ }\mathbb C)$.

Take $P \in GL_r(\mathbb C)$ such that $\Lambda=PAP^{-1}$ is a Jordan normal form.
Set $\Lambda=\Lambda_1 \oplus \cdots \oplus \Lambda_t$, where $\Lambda_i$ is 
a Jordan block of size $(l_i +1) \times (l_i+1)$ with eigenvalue $\lambda_i$.
Set $\rho_i=|\lambda_i|$, $\rho=\rho(A)$, $l=l(A)$.
We may assume that, in the lexicographic order,
$(\rho_i,l_i)=(\rho,l)$ for $1 \leq i \leq s$ and $(\rho_i,l_i) < (\rho,l)$ for $s+1 \leq i \leq t$.
We will prove the theorem in the cases when $\omega=0$ and $\omega \neq 0$,
respectively.

\underline{\textbf{The $\omega=0$ case.}}

This is the case when $\End(E)_{\mathbb Q}=\mathbb Q$.
We extend $\langle \cdot,\cdot \rangle_E$ to a hermitian form on $V_{\mathbb C}$ as 
$\langle x,\alpha y \rangle_E=\alpha \langle x,y \rangle_E
=\langle \overline \alpha x, y \rangle_E$ for $x,y \in M$, $\alpha \in \mathbb C$.
Then $\hat h_E, \langle \cdot,\cdot \rangle_X, \hat h_X$ are also extended and 
$\langle \cdot,\cdot \rangle_E, \langle \cdot,\cdot \rangle_X$ are positive definite 
hermitian forms.
We define $\mathbb C$-linear maps $F, G, \Phi: V_{\mathbb C}^r \to V_{\mathbb C}^r$ 
as $F(x)=Ax$, $G(x)=\Lambda x$ and $\Phi(x)=Px$.
Then $\Phi \circ F=G \circ \Phi$.
Take any $y=(y_1,\ldots,y_t) \in V_{\mathbb C}^r$,
where $y_i=(y_{i,0},\ldots,y_{i,l_i})$, $y_{i,j} \in V_{\mathbb C}$.
Then 
$$
\begin{cases}
G^n(y)=(\Lambda_1^n y_1,\ldots,\Lambda_t^n y_t), \\
\Lambda_i^n y_i=(\sum_{k=0}^{l_i} \binom{n}{k} \lambda_i^{n-k} y_{i,k},
\sum_{k=0}^{l_i-1} \binom{n}{k} \lambda_i^{n-k} y_{i,k+1},\ldots,\lambda_i^n y_{i,l_i}).
\end{cases}
$$
So 
\begin{align*}
\hat h_X(G^n(y))
&= \sum_{i=1}^t \sum_{j=0}^{l_i} 
\hat h_E \left( \sum_{k=0}^{l_i-j} \binom{n}{k} \lambda_i^{n-k}y_{i,k+j} \right) \\
&= \sum_{i=1}^t \sum_{j=0}^{l_i} \sum_{k,k'=0}^{l_i-j} \binom{n}{k} \binom{n}{k'}
\rho_i^{2n} \langle \lambda_i^{-k} y_{i,k+j}, \lambda_i^{-k'} y_{i,k'+j} \rangle_E \\
&= \sum_{i=1}^s \binom{n}{l}^2 \rho^{2n-2l} \hat h_E(y_{i,l}) + o(\rho^{2n} n^{2l})
\ \ (n \to \infty).
\end{align*}
This implies that 
$$\lim_{n \to \infty} \frac{\hat h_X(G^n(y))}{\rho^{2n}n^{2l}}
= \frac{1}{(l! \rho^l)^2} \sum_{i=1}^s \hat h_E(y_{i,l}).$$

We prepare the following easy lemma.
\begin{lem}\label{lem5.3}
Let $\langle \cdot,\cdot \rangle, \langle \cdot,\cdot \rangle'$ be 
positive definite hermitian (resp.~quadratic) forms 
on a finite dimensional $\mathbb C$-vector 
space (resp.~$\mathbb R$-vector space) $W$.
Set $f(x)=\langle x,x \rangle$, $g(x)=\langle x,x \rangle'$.
Then $f \asymp g$.
\end{lem}

\begin{proof}
Introduce a norm $||\cdot||$ on $W$ and let $S$ be the subset of $W$ 
consisting of the elements of norm 1.
Since $f/g$ is a non-vanishing continuous function on the compact space $S$,
we can take $A,B>0$ such that $A \leq f(x)/g(x) \leq B$ for $x \in S$.
Here $f(ax)/g(ax)=f(x)/g(x)$ for $x \in W  \setminus \{0\}$ and 
$a \neq 0$. 
So $Ag(x) \leq f(x) \leq Bg(x)$ for $x \in W$.
Thus $f \asymp g$.
\end{proof}

Fix a number field $K$ where all concerned are defined.
Since $\hat h_X, \hat h_X \circ \Phi^{-1}$ are positive definite hermitian forms on 
the finite dimensional $\mathbb C$-vector space $X(K)_{\mathbb C}$,
we can take $C_1,C_2>0$ such that 
$C_1 \hat h_X \leq \hat h_X \circ \Phi^{-1} \leq C_2 \hat h_X$ on $X(K)_{\mathbb C}$ 
by Lemma \ref{lem5.3}.
Take any $x \in X(K)$. Then $f^n(x)=F^n(x)=\Phi^{-1}G^n\Phi(x)$. So 
$$C_1 \frac{\hat h_X(G^n\Phi(x))}{\rho^{2n}n^{2l}} \leq \frac{\hat h_X(f^n(x))}{\rho^{2n}n^{2l}} 
\leq C_2 \frac{\hat h_X(G^n\Phi(x))}{\rho^{2n}n^{2l}}.$$
Set $\hat h_X^+=\hat h_X+1$, then $\hat h_X^+ \asymp h_X$.
Represent $\Phi(x)$ as 
$$\Phi(x)=(\Phi_{1,0}(x),\ldots,\Phi_{1,l_1}(x),\ldots,\Phi_{t,0}(x),\ldots,\Phi_{t,l_t}(x)).$$
By the above calculation, we have 
$$\frac{C_1}{(l!\rho^l)^2} \sum_{i=1}^s \hat h_E(\Phi_{i,l}(x)) 
\leq \liminf_{n \to \infty} \frac{\hat h_X^+(f^n(x))}{\rho^{2n}n^{2l}}$$
and
$$\limsup_{n \to \infty} \frac{\hat h_X^+(f^n(x))}{\rho^{2n}n^{2l}}
\leq \frac{C_2}{(l!\rho^l)^2} \sum_{i=1}^s \hat h_E(\Phi_{i,l}(x)).$$
Now $\Phi_{i,l}: V_{\mathbb C}^r \to V_{\mathbb C}$ is not identically zero for each $i$ 
since $\Phi:V_{\mathbb C}^r \to V_{\mathbb C}^r$ is an automorphism.
Hence $\sum_{i=1}^s \hat h_E \circ \Phi_{i,l}$ is not identically zero 
on $X(\overline{\mathbb Q})$.
So Theorem \ref{thm3.12} implies that $(\delta_f,l_f)=(\rho^2,2l)$ and 
$\overline h_f \asymp \underline h_f \asymp \sum_{i=1}^s \hat h_E \circ \Phi_{i,l}$.

\underline{\textbf{The $\omega \neq 0$ case.}} 

This is the case when $E$ has complex multiplication.
We extend $\langle \cdot,\cdot \rangle_E$ to a quadratic form on $V_{\mathbb R}$ as 
$\langle x,\alpha y \rangle_E=\alpha \langle x,y \rangle_E
=\langle \alpha x, y \rangle_E$ for $x,y \in M$, $\alpha \in \mathbb R$.
Then $\hat h_E, \langle \cdot,\cdot \rangle_X, \hat h_X$ are also extended and 
$\langle \cdot,\cdot \rangle_E, \langle \cdot,\cdot \rangle_X$ are positive definite 
quadratic forms.
Since $\mathbb Q(\omega)$ acts on $V_{\mathbb Q}$ and 
$\mathbb Q(\omega) \otimes_{\mathbb Q} \mathbb R = \mathbb R(\omega)=\mathbb C$,
$V_{\mathbb R}$ already has a structure of $\mathbb C$-vector space.
We define $\mathbb C$-linear maps $F, G, \Phi: V_{\mathbb R}^r \to V_{\mathbb R}^r$ 
as $F(x)=Ax$, $G(x)=\Lambda x$ and $\Phi(x)=Px$.
Here we will make some lemmas.

\begin{lem}\label{lem5.5}
Let $W$ be a smooth projective variety, $f,g$ endomorphisms on $W$ with 
$\delta_f,\delta_g >1$ and $f \circ g = g \circ f$.
Let $D$ be an $\mathbb R$-divisor on $W$ such that $D \not\equiv 0$,
$f^*D \equiv \alpha D$, $g^*D \equiv \beta D$ for some $\alpha > \sqrt{\delta_f}$ and 
$\beta > \sqrt{\delta_g}$.
Then $\hat h_{D,f}=\hat h_{D,g}$.
\end{lem}

\begin{proof}
Since $h_D \circ f=h_D+O(\sqrt{h_W})$, we can take $C>0$ such that 
$|h_D \circ f - h_D| \leq C\sqrt{h_W}$.
Take $\varepsilon >0$ such that $\delta_f+\varepsilon < \delta_f^2$ and 
$\delta_g+\varepsilon < \delta_g^2$, and use Theorem \ref{thm_KSM} to take 
$C'>0$ such that 
$h_W \circ f^n \leq C' (\delta_f+\varepsilon)^n h_W$ and 
$h_W \circ g^n \leq C' (\delta_g+\varepsilon)^n h_W$ for every $n$.
We compute 
\begin{align*}
\left| \frac{h_D \circ f^k}{\alpha^k}-\frac{h_D \circ f^{k-1}}{\alpha^{k-1}} \right|
&\leq C \frac{\sqrt{h_W \circ f^{k-1}}}{\alpha^k} \\
&\leq C\sqrt{C'} \frac{\sqrt{(\delta_f+\varepsilon)^{k-1} h_W}}{\alpha^k} \\
&\leq C\sqrt{C'} \left( \frac{\delta_f+\varepsilon}{\alpha^2} \right)^{k/2}  \sqrt{h_W}.
\end{align*}
So we obtain 
\begin{align*}
\left| \frac{h_D \circ f^n}{\alpha^n}-h_D\right|
&\leq \sum_{k=1}^n 
\left| \frac{h_D \circ f^k}{\alpha^k}-\frac{h_D \circ f^{k-1}}{\alpha^{k-1}} \right| \\
&\leq C\sqrt{C'} \sum_{k=1}^n 
\left( \frac{\delta_f+\varepsilon}{\alpha^2} \right)^{k/2} \sqrt{h_W} \\
&\leq C'' \sqrt{h_W},
\end{align*}
where $C''=C\sqrt{C'} \sum_{k=1}^\infty ((\delta_f+\varepsilon)/\alpha^2)^{k/2}$,
and 
\begin{align*}
\left| \frac{h_D \circ f^n \circ g^n}{\alpha^n \beta^n}-\frac{h_D \circ g^n}{\beta^n} \right|
&\leq C'' \frac{\sqrt{h_W \circ g^n}}{\beta^n} \\
&\leq C'' \frac{\sqrt{C'(\delta_g +\varepsilon)^n h_W}}{\beta^n} \\
&= C''\sqrt{C'} \left( \frac{\delta_g+\varepsilon}{\beta^2} \right)^{n/2} \sqrt{h_W}.
\end{align*}
So 
$$\hat h_{D,f \circ g}=\lim_{n \to \infty} \frac{h_D \circ f^n \circ g^n}{\alpha^n \beta^n}
= \lim_{n \to \infty} \frac{h_D  \circ g^n}{ \beta^n} = \hat h_{D,g}.$$
Similarly $\hat h_{D,f \circ g}= \hat h_{D,f}$.
So we obtain $\hat h_{D,f}=\hat h_{D,g}$.
\end{proof}

\begin{lem}\label{lem5.6}
Let $\phi \in \End(E)$ be an isogeny with $\delta_{\phi} >1$.
Then $\hat h_E \circ \phi =\delta_{\phi} \hat h_E$.
\end{lem}

\begin{proof}
Let $H$ be a symmetric ample divisor on $E$.
Then $[2]^*H \sim 4H$ and $f^*H \equiv \deg(f) H = \delta_f H$.
Since $f$ is a group homomorphism, $f \circ [2] = [2] \circ f$.
So $\hat h_E=\hat h_{H, [2]}=\hat h_{H,f}$ by Lemma \ref{lem5.5} and hence 
$\hat h_E \circ f=\hat h_{H,f} \circ f=\delta_f \hat h_{H,f}=\delta_f \hat h_E$.
\end{proof}

\begin{lem}\label{lem5.7}
We have $\langle \alpha x,\alpha y \rangle_E= |\alpha |^2 \langle x,y \rangle_E$ for 
$x,y \in V_{\mathbb R}$, $\alpha \in \mathbb C$.
\end{lem}

\begin{proof}
We can take a positive integer $m$ such that 
$\phi=m\omega \in \End(E)$.
Then $\phi^2=[-m^2d]$, so $\deg(\phi)=\sqrt{\deg(\phi^2)}=\sqrt{m^4d^2}=m^2d$.
By Lemma \ref{lem5.6}, $\hat h_E \circ \phi= m^2d \hat h_E$.
For any $x,y \in E(\overline{\mathbb Q})$,
\begin{align*}
\langle \phi(x),\phi(y) \rangle_E
&=\frac{1}{2} \left(\hat h_E(\phi(x)+\phi(y))-\hat h_E(\phi(x))-\hat h_E(\phi(y)) \right) \\
&=\frac{1}{2} \left(m^2d \hat h_E(x+y)-m^2d \hat h_E(x)-m^2d \hat h_E(y) \right) \\
&=m^2d \langle x,y \rangle_E.
\end{align*}
By linearity, $\langle \phi(x),\phi(y) \rangle_E= m^2d \langle x,y \rangle_E$ holds 
for $x,y \in V_{\mathbb R}$.
Then  
$\langle \omega x,\omega y \rangle_E  =m^{-2} \langle \phi(x),\phi(y) \rangle_E 
= d \langle x,y \rangle_E$ for $x,y \in V_{\mathbb R}$.

Take any $\alpha \in \mathbb C$.
Set $\alpha=\alpha_1+ \alpha_2 \omega$, $\alpha_1,\alpha_2 \in \mathbb R$.
Then 
\begin{align*}
\langle \alpha x,\alpha y \rangle_E
&= \langle \alpha_1x, \alpha_1y \rangle_E+\langle \alpha_1x, \alpha_2 \omega y \rangle_E
+ \langle \alpha_2 \omega x, \alpha_1 y \rangle_E
+ \langle \alpha_2 \omega x, \alpha_2 \omega y \rangle_E \\
&= \alpha_1^2 \langle x,y \rangle_E + \alpha_1\alpha_2 \langle x,\omega y \rangle_E
+ \alpha_1 \alpha_2 \langle \omega x,y \rangle_E+ \alpha_2^2d \langle x,y \rangle_E \\
&= |\alpha|^2 \langle x,y \rangle_E
+ \alpha_1 \alpha_2(\langle x, \omega y \rangle_E + \langle \omega x, y \rangle_E).
\end{align*}
Here $\langle x, \omega y \rangle_E = d^{-1}\langle \omega x, \omega^2 y \rangle_E
= d^{-1}\langle \omega x, -dy \rangle_E = -\langle \omega x, y \rangle_E$,
so $\langle \alpha x,\alpha y \rangle_E= |\alpha|^2 \langle x,y \rangle_E$.
\end{proof}

Take $y \in V^r_{\mathbb R}$.
We use the notation as in the $\omega=0$ case.
Using Lemma \ref{lem5.7}, we compute  
\begin{align*}
\hat h_X(G^n(y))
&= \sum_{i=1}^t \sum_{j=0}^{l_i} 
\hat h_E \left( \sum_{k=0}^{l_i-j} \binom{n}{k} \lambda_i^{n-k}y_{i,k+j} \right) \\
&= \sum_{i=1}^t \sum_{j=0}^{l_i} \sum_{k,k'=0}^{l_i-j} \binom{n}{k} \binom{n}{k'}
\rho_i^{2n} \langle \lambda_i^{-k} y_{i,k+j}, \lambda_i^{-k'} y_{i,k'+j} \rangle_E \\
&= \sum_{i=1}^s \binom{n}{l}^2 \rho^{2n-2l} \hat h_E(y_{i,l}) + o(\rho^{2n} n^{2l})
\ \ (n \to \infty).
\end{align*}
This implies that 
$$\lim_{n \to \infty} \frac{\hat h_X(G^n(y))}{\rho^{2n}n^{2l}}
= \frac{1}{(l! \rho^l)^2} \sum_{i=1}^s \hat h_E(y_{i,l}).$$
Then we obtain $(\delta_f,l_f)=(\rho^2,2l)$ as in the $\omega=0$ case.
\end{proof}


\section{Automorphisms on surfaces}\label{sec6}

In this section, we study automorphisms on surfaces.
Kawaguchi \cite{Kaw08} constructed nef canonical heights for an automorphism 
and its inverse, and proved that the zero set of the sum of those heights is the union 
of the periodic curves and the periodic points (\cite[Theorem 5.2]{Kaw08}).
The arithmetic degrees for automorphisms on surfaces are well understood by 
 Kawaguchi--Silverman \cite{KaSi14}.
We compute the ample canonical heights for automorphisms on surfaces in this section.
However,
most of the computations here are essentially contained in Kawaguchi \cite{Kaw08} 
and Kawaguchi--Silverman \cite{KaSi14}.

As a related result, Jonsson--Reschke \cite{JoRe15} proved that 
a nef canonical height for a birational surface self-map converges at every point with 
well-defined forward orbit.
As we will see in Theorem \ref{thm6.6} below, such a nef canonical height is equivalent 
to the upper and lower ample canonical heights if the self-map is an automorphism.

Our aim in this section is to prove the following 
(cf.~{\cite[Theorem 5.2]{Kaw08}} and {\cite[Theorem 9, 10]{KaSi14}}).

\begin{thm}\label{thm6.6}
Let $X$ be a surface and $f$ an automorphism on $X$ with $\delta_f>1$.

\begin{itemize}
\item[(i)]
We have $l_f=0$, and there is a nef canonical height $\hat h^+$ such that 
$\overline h_f \asymp \underline h_f \asymp \hat h^+$.
\item[(ii)]
Take $x \in X(\overline{\mathbb Q})$.
Then the following are equivalent.
\begin{itemize}
\item[(1)]
$O_f(x)$ is dense.
\item[(2)]
$\overline h_f(x)>0$.
\item[(3)]
$\underline h_f(x)>0$.
\item[(4)]
$\alpha_f(x)=\delta_f$.
\end{itemize}
Moreover, if $\alpha_f(x)<\delta_f$, then $\alpha_f(x)=1$.
\item[(iii)]
Let $\mathcal C=\{C_i\}$ be the set of $f$-periodic irreducible curves on $X$.
Then $Z_f(\overline{\mathbb Q})
=\Per_f(\overline{\mathbb Q}) \cup \bigcup_i C_i(\overline{\mathbb Q})$,
and 
$\Per_f(K) \setminus \bigcup_i C_i(K)$ 
is finite for any number field $K$.
\end{itemize}
\end{thm}

First, we prepare some lemmas.
The following lemma follows from the Hodge index theorem 
(cf.~{\cite[Lemma 1.2 (3)]{Kaw08}}).

\begin{lem}\label{lem6.1}
Let $X$ be a surface and $v_1,v_2 \in N^1(X)_{\mathbb R} \setminus \{0\}$ be 
nef classes which are linearly independent.
Then $v_1+v_2$ is nef and big.
\end{lem}

\begin{defn}\label{defn6.2}
Let $X$ be a surface and $v \in N^1(X)_{\mathbb R}$ a  class on $X$.
We set $Z(v)=\{ C \mid  C \mathrm{\ is\ an\ irreducible\ curve\ on\ }$X$ 
\mathrm{\ with\ } (C \cdot v)=0 \}$. 
\end{defn}

\begin{lem}[{\cite[Proposition 1.3]{Kaw08}}]\label{lem6.3}
Let $X$ be a surface and $v \in N^1(X)_{\mathbb R}$ a nef and big class on $X$.
\begin{itemize}
\item[(i)]
$Z(v)$ is a finite set. 
\item[(ii)]
There is an effective divisor $Z$ on $X$ such that 
$\Supp Z= \bigcup_{C \in Z(v)}C$ and $v-\varepsilon Z$ is ample for sufficiently small 
$\varepsilon >0$.
\end{itemize}
\end{lem}

\begin{lem}\label{lem6.4}
Let $X$ be a surface and $f$ an automorphism on $X$.
Then $\delta_{f^{-1}}=\delta_f$.
\end{lem}

\begin{proof}
Take an ample divisor $H$ on $X$.
Then 
$\delta_{f^{-1}}=\lim_{n \to \infty} ((f^{-n})^*H \cdot H)^{1/n}
=\lim_{n \to \infty} ((f^n)_*H \cdot H)^{1/n}
=\lim_{n \to \infty} (H \cdot (f^n)^*H)^{1/n}
=\delta_f$.
\end{proof}

\begin{lem}\label{lem6.5}
Let $X$ be a surface, $f$ an automorphism on $X$ with $\delta_f>1$,
and $D^+,D^-$ be nef $\mathbb R$-divisors such that $D^+,D^- \not\equiv 0$ and 
$f^*D^+ \equiv \delta_f D^+$, $(f^{-1})^*D^- \equiv \delta_f D^-$.
Set $D=D^++D^-$.
Then $D$ is nef and big and,
for any irreducible curve $C$ on $X$,
$C \in Z(D)$ if and only if $C$ is $f$-periodic.
\end{lem}

\begin{proof}
$D^+,D^-$ are linearly independent in $N^1(X)_{\mathbb R}$ since they are eigenvectors 
with different eigenvalues.
So $D$ is nef and big by Lemma \ref{lem6.1}.

Let $C$ be any irreducible curve on $X$.
Note that $(C \cdot D)=0$ if and only if $(C \cdot D^+)=(C \cdot D^-)=0$.
Assume that $C \in Z(D)$.
Then $(D^+ \cdot f(C))=(D^+ \cdot f_*C)=(f^*D^+ \cdot C)=\delta_f(D^+ \cdot C)=0$.
Similarly $(D^- \cdot f(C))=0$, so $f(C) \in Z(D)$.
Lemma \ref{lem6.3} implies that $Z(D)$ is finite.
Since $C,f(C),f^2(C),\ldots \in Z(D)$, 
it follows that $f^k(C)=f^l(C)$ for some $k<l$.
Then $C=f^{l-k}(C)$ since $f$ is an automorphism.
Thus $C$ is $f$-periodic.

Conversely, assume that $C$ is $f$-periodic.
Then $f^N(C)=C$ for some $N \in \mathbb Z_{>0}$.
So $(D^+ \cdot C)=(D^+ \cdot (f^N)_*C)=((f^N)^* D^+ \cdot C)=\delta_f^N(D^+ \cdot C)$.
Since $\delta_f^N>1$, $(D^+ \cdot C)$ must be zero.
Similarly $(D^- \cdot C)=0$, so $C \in Z(D)$.
\end{proof}

\begin{proof}[Proof of Theorem \ref{thm6.6}]
(i) Take  nef $\mathbb R$-divisors $D^+,D^-$ on $X$ such that 
$D^+,D^- \not\equiv 0$, $f^*D^+ \equiv \delta_f D^+$ and 
$(f^{-1})^*D^- \equiv \delta_f D^-$.
Set $D=D^++D^-$.
Lemma \ref{lem6.5} implies that  $\mathcal C=Z(D)$. 
By Lemma \ref{lem6.3},  
$\mathcal C$ is a finite set and we can take $a_i>0$ for each $C_i \in \mathcal C$ 
such that, setting $E=\sum_i a_i C_i$, $A=D-E$ is ample.
Set $\hat h^+=\hat h_{D^+,f}$, $\hat h^-=\hat h_{D^-,f^{-1}}$.
Take a height $h_A$ associated to $A$ as satisfying $h_A \geq 1$.
Then $h_A=h_D-h_E+O(1)=\hat h^+ +\hat h^- -h_E+O(\sqrt{h_A})$.
Set $\phi=h_A-\hat h^+ -\hat h^- +h_E$.
We have 
\begin{align*}
\frac{h_A \circ f^n}{\delta_f^n}
&= \frac{\hat h^+ \circ f^n}{\delta_f^n} +\frac{\hat h^- \circ f^n}{\delta_f^n}
-\frac{h_E \circ f^n}{\delta_f^n} +\frac{\phi \circ f^n}{\delta_f^n} \\
&= \hat h^+ +\frac{\hat h^-}{\delta_f^{2n}}
-\frac{h_E \circ f^n}{\delta_f^n} +\frac{\phi \circ f^n}{\delta_f^n}.
\end{align*}
Lemma \ref{lem4.3} (i) implies that $\lim_n \delta_f^{-n} \phi(f^n(x))=0$ for every $x$.
Since every irreducible component of $E$ is $f$-periodic, $(f^N)^*E \sim E$ for 
some $N \in \mathbb Z_{>0}$.
So, applying Lemma \ref{lem4.3} (ii),
$\lim_n \delta_f^{-Nn} h_E(f^{Nn}(x))=0$ 
for every $x$.
This implies that $\lim_n \delta_f^{-n} h_E(f^{n}(x))=0$ for every $x$.
Therefore $\lim_n \delta_f^{-n} h_A \circ f^n=\hat h^+$.
So $l_f=0$ and $\overline h_f, \underline h_f \asymp \hat h^+$.

(ii) Assume (1).
Take $C>0$ such that $h_A \leq \hat h^+ + \hat h^- -h_E +C\sqrt{h_A}$.
Then 
\begin{align*}
\sqrt{h_A(f^n(x))}(\sqrt{h_A(f^n(x))}-C) 
&\leq \hat h^+(f^n(x))+\hat h^-(f^n(x))-h_E(f^n(x)) \\
&\leq \delta_f^n \hat h^+(x) + \delta_f^{-n} \hat h^-(x) -h_E(f^n(x)).
\end{align*}
So $\{ \delta_f^n \hat h^+(x) + \delta_f^{-n} \hat h^-(x) -h_E(f^n(x)) \}_{n=0}^\infty$ is 
upper unbounded by the Northcott finiteness theorem.
Since $O_f(x)$ is dense, the set $O_f(x) \cap \Supp E$ is finite due to the dynamical 
Mordell--Lang theorem for \'etale endomorphisms (cf.~\cite[Corollary 1.4]{BGT10}).
Moreover, $-h_E$ is upper bounded on $X \setminus \Supp E$.
Therefore $\{ -h_E(f^n(x)) \}_{n=0}^\infty$ is upper bounded.
Then $\hat h^+(x)$ must be positive.
So $\overline h_f(x)>0$ since $\overline h_f \asymp \hat h^+$.

(2) is equivalent to (3) because $\overline h_f \asymp \underline h_f$.

(3) implies (4) by Proposition \ref{prop3.2} (iv).

Finally, assume that $O_f(x)$ is not dense and we show that $\alpha_f(x)=1$.
Let $Z$ be the Zariski closure of $O_f(x)$.
If $\dim Z=0$, then $x$ is $f$-preperiodic and so $\alpha_f(x)=1$.
Assume that $\dim Z=1$.
We have $f(Z)=Z$ since $f(Z) \subset Z$ and $f$ is an automorphism.
So $f|_Z$ is an automorphism on $Z$.
Replacing $f$ by a power, we may assume that $f(Z_i)=Z_i$ for every irreducible component 
$Z_i$ of $Z$. So we may assume that $Z$ is irreducible.
Take the normalization $C$ of $Z$ and let $\nu: C \to X$ be the induced morphism.
Then $f|_Z$ induces an automorphism $g$ on $C$ such that $\nu \circ g=f \circ \nu$.
Since $\nu$ is finite, $h_C \asymp h_X \circ \nu$.
So, taking $x_0 \in \nu^{-1}(x)$, we have 
$\alpha_f(x)=\lim_{n} h_X(f^n(x))^{1/n}=\lim_{n}h_X(\nu g^n(x_0))^{1/n}
=\lim_{n} h_C(g^n(x_0))^{1/n}=\alpha_g(x_0) \leq \delta_g=1$,
where $\delta_g=1$ because $g$ is an automorphism on a curve.
Therefore $\alpha_f(x)=1$.

(iii) 
By Proposition \ref{prop3.2} (ii), $\Per_f(\overline{\mathbb Q})
=\Preper_f(\overline{\mathbb Q}) \subset Z_f(\overline{\mathbb Q})$.
For any $x \in \bigcup_i C_i(\overline{\mathbb Q})$, 
we have $O_f(x) \subset \bigcup_i C_i(\overline{\mathbb Q})$ and so it is not dense.  
Then $\underline h_f(x)=0$ by (ii).
Thus $\bigcup_i C_i(\overline{\mathbb Q}) \subset Z_f(\overline{\mathbb Q})$.
Conversely, take any $x \in Z_f(\overline{\mathbb Q})$.
Then $O_f(x)$ is not dense by (ii).
Let $W$ be the closure of $O_f(x)$.
Then $\dim W \leq 1$ and $f(W) \subset W$.
So each irreducible component of $W$ is an $f$-periodic curve or an $f$-periodic point,
which implies that $x \in W(\overline{\mathbb Q}) \subset 
\Per_f(\overline{\mathbb Q}) \cup \bigcup_i C_i(\overline{\mathbb Q})$.
Thus $Z_f(\overline{\mathbb Q})
=\Per_f(\overline{\mathbb Q}) \cup \bigcup_i C_i(\overline{\mathbb Q})$.

Take $M>0$ such that $-h_E \leq M$ on 
$X \setminus \Supp E=X \setminus \bigcup_i C_i$.
Take any 
$x \in \Per_f(\overline{\mathbb Q}) \setminus \bigcup_i C_i(\overline{\mathbb Q})$.
Then $x$ is also $f^{-1}$-periodic since $f$ is an automorphism.
Moreover, we have 
$O_f(x) \cap \left( \bigcup_i C_i(\overline{\mathbb Q}) \right) =\varnothing$ 
since  $x \not\in \bigcup_i C_i$ and 
$\{ C_i \}_i$ are the whole of $f^{-1}$-periodic curves.
So the inequality 
$\sqrt{h_A}(\sqrt{h_A}-C) \leq \hat h^+ + \hat h^- -h_E$ 
implies that 
$\sqrt{h_A(x)}(\sqrt{h_A(x)}-C) \leq M$ for 
$x \in \Per_f(\overline{\mathbb Q}) \setminus \bigcup_i C_i(\overline{\mathbb Q})$.
Then the Northcott finiteness theorem shows that 
$\Per_f(K) \setminus \bigcup_i C_i(K)$ 
is finite for any number field $K$.
\end{proof}


\section{Non-trivial endomorphisms on surfaces}\label{sec7}

The aim in this section is to prove the following.

\begin{thm}\label{thm7.7}
Let $X$ be a surface and $f$ a non-trivial endomorphism on $X$ with $\delta_f>1$.
Then Conjecture \ref{conj_zf} holds for $f$.
Moreover, if $X$ is not birational to an abelian surface, then $l_f=0$.
\end{thm}

To prove it, we will give some lemmas.

\begin{lem}[{\cite[Lemma 3.3]{MSS17}}]\label{lem7.1}
Let $X,Y$ be smooth projective varieties, $\mu: X \dashrightarrow Y$ a birational map,
and $U \subset X$ an open subset of $X$ such that 
$\mu|_U: U \to \mu(U)$ is an isomorphism.
Then $h_X|_U \asymp (h_Y \circ \mu)|_U$.
\end{lem}

\begin{lem}\label{lem7.2}
Let $X$ be a surface, $E$ a $(-1)$-curve on $X$, $\mu: X \to Y$ the contraction of $E$,
$f$ an endomorphism on $X$ with $f(E)=E$, and $g$ an endomorphism on $Y$ such that 
$\mu \circ f=g \circ \mu$.
\begin{itemize}
\item[(i)]
$(\delta_f,l_f)=(\delta_g,l_g)$.
\item[(ii)]
Let $K \subset \overline{\mathbb Q}$ be any subfield where 
all concerned are defined.
Then 
$Z_f(K) \subset \mu^{-1}(Z_g(K))$.
\end{itemize}
\end{lem}

\begin{proof}
(i) It follows from the product formula (cf.~\cite{Tru15}) that $\delta_f=\delta_g$.
If $C$ is an irreducible curve on $X$ such that $f(C)=E=f(E)$,
then $f_*C=af_*E$ for some $a>0$.
Now we have the equation $f_* \circ f^*=\deg(f) \id_{N^1(X)_{\mathbb R}}$.
This implies that $f^*$ and $f_*$ are automorphisms on $N^1(X)_{\mathbb R}$.
So $C \equiv aE$ since $f_*$ is injective.
Hence $f^*E \equiv dE$ for some $0< d \leq \delta_f$.
Take an ample divisor $H_Y$ on $Y$.
Then $\mu^*H_Y$ is nef and big, and $H_X=\mu^*H_Y -bE$ is ample 
for some $b>0$.
Take any non-negative integer $l$.
Then 
$$\frac{h_{H_X} \circ f^n}{\delta_f^n n^l}
= \frac{h_{H_Y} \circ \mu \circ f^n-b h_E \circ f^n}{\delta_f^n n^l}
=\frac{h_{H_Y} \circ g^n \circ \mu}{\delta_g^n n^l}-b \frac{h_E \circ f^n}{\delta_f^n n^l}.$$
By Lemma \ref{lem4.3} (ii),
$\limsup_n \delta_f^{-n} |h_E(f^n(x))| < \infty$ for every $x$.
So $l_f=l_g$.

(ii) Take any $x \in Z_f(K)$.
If $O_f(x) \cap E \neq \varnothing$, then $\mu(x)$ is $g$-preperiodic and so 
$\mu(x) \in Z_g(K)$.
Assume that $O_f(x) \cap E = \varnothing$.
By Lemma \ref{lem7.1},
we have $h_X|_{X \setminus E} \asymp (h_Y \circ \mu)|_{X \setminus E}$.
So $\underline h_f(x)=0$ implies $\underline h_g(\mu(x))=0$ by (i).
Thus $Z_f(K) \subset \pi^{-1}(Z_g(K))$.
\end{proof}

\begin{lem}[{\cite[Proposition 10]{Nak02}}]\label{lem7.3}
Let $X$ be a surface and $f$ a non-trivial endomorphism on $X$.
Then there is a positive integer $N$ such that 
$f^N(C)=C$ for every irreducible curve $C$ on $X$ with negative self-intersection.
\end{lem}

As a result of \cite{MSS17}, we have the following.

%
%
%
%

\begin{thm}[{\cite{MSS17}}]\label{thm7.4}
Let $X$ be a surface and $f$ a non-trivial endomorphism on $X$ with $\delta_f>1$.
Assume that $X$ has no $(-1)$-curve and isomorphic to neither $\mathbb P^2$ nor 
abelian surfaces.
Consider the following two operations to $(X,f)$.

{\rm (a):} 
$X'$ is a surface, $f'$ is an endomorphism on $X'$, and 
$\phi:X' \to X$ is an \'etale morphism 
such that $\phi \circ f'=f \circ \phi$.
Replace $(X,f)$ by $(X',f')$.

{\rm (b):} 
Replace $(X,f)$ by $(X,f^N)$ for a positive integer $N$.

After applying (a) and (b) to $(X,f)$ finite times,
$(X,f)$ falls into one of the following two types.

{\rm (I):}  
There are a curve $C$ and a surjective morphism $\pi: X \to C$ such that $\pi \circ f=\pi$.
$$
\xymatrix{
X \ar[dr]_{\pi} \ar[rr]^{f} & & X \ar[dl]^{\pi} \\ 
 & C & 
}
$$
{\rm (II):} 
There are a curve $C$, an endomorphism $g$ on $C$ with $\delta_g=\delta_f$,
and a surjective morphism $\pi: X \to C$ such that $\pi \circ f=g \circ \pi$.
$$
\xymatrix{
X \ar[d]_{\pi} \ar[r]^{f} & X \ar[d]^{\pi} \\ 
C  \ar[r]^{g}   & C
}
$$

More precisely, under the assumption, 
$X$ is isomorphic to 
a $\mathbb P^1$-bundle, a bielliptic surface, or a properly elliptic surface 
(a minimal surface of Kodaira dimension one), and 
\begin{itemize}
\item
$\mathbb P^1$-bundles over a curve of genus $\geq 2$ and properly elliptic surfaces 
fall into type {\rm (I)}. 
\item
Hirzebruch surfaces, $\mathbb P^1$-bundles over elliptic curves, and bielliptic surfaces 
fall into type {\rm (II)}. 
\end{itemize}
\end{thm}

Lemma \ref{lem7.5} and Lemma \ref{lem7.6} below treat the type (I) and (II), respectively.

\begin{lem}\label{lem7.5}
Let $X$ be a surface and $f$ an endomorphism on $X$ such that 
$\delta_f>1$.
Let $C$ be a curve and $\pi: X \to C$ a surjective morphism such that $\pi \circ f=\pi$.
Then $l_f=0$.
\end{lem}

\begin{proof}
Take any $x \in X(\overline{\mathbb Q})$ and set $F=\pi^{-1}(\pi(x))$.
Since $f|_F$ permutes the irreducible components of $F$,
replacing $f$ by a power, we may assume that $f$ preserves any irreducible components 
of $F$.
So $O_f(x) \subset F_1$ for some irreducible component $F_1$ of $F$.
Take the normalization $Z$ of $F_1$ and let $\nu: Z \to X$ be the induced morphism.
Take $x_0 \in \nu^{-1}(x)$.
$f|_{F_1}: F_1 \to F_1$ induces an endomorphism $g$ on $Z$ such that 
$\nu \circ g=f \circ \nu$.
Take an ample divisor $H$ on $X$.
Then we can take $M>0$ such that 
$(\nu_*Z \cdot D) \leq M(H \cdot D)$ for any nef $\mathbb R$-divisor $D$.
So
\begin{align*}
\delta_g
&=\lim_{n \to \infty}\deg(g^{n*}\nu^*H)^{1/n}
=\lim_{n \to \infty}(\nu_*Z \cdot f^{n*}H)^{1/n} \\
&\leq \lim_{n \to \infty}(M(H \cdot f^{n*}H))^{1/n}
=\delta_f.
\end{align*}
We have $h_X \circ \nu \asymp h_Z$ since $\nu$ is finite.
Hence 
\begin{align*}
\overline h_{f,\delta_f,0}(x)
&= \limsup_{n \to \infty} \frac{h_X(f^n(x))}{\delta_f^n} \\
&= \limsup_{n \to \infty} \frac{h_X(f^n \nu(x_0))}{\delta_f^n} \\
&\asymp \limsup_{n \to \infty} \left( \frac{\delta_g}{\delta_f} \right)^n 
\frac{h_Z(g^n(x_0))}{\delta_g^n} 
< \infty.
\end{align*}
Note that $g$ is an endomorphism on a curve and so  $l_g=0$ if 
$\delta_g>1$ due to Theorem \ref{thm4.1}.
Then it follows that $l_f=0$.
\end{proof}

\begin{lem}\label{lem7.6}
Let $X$ be a surface and $f$ an endomorphism on $X$ such that 
$\delta_f>1$.
Let $C$ be a curve, $g$ an endomorphism on $C$,
and $\pi: X \to C$ a surjective morphism such that $\pi \circ f=g \circ \pi$.
Assume that $(\delta_f,l_f)=(\delta_g,l_g)$.
Then Conjecture \ref{conj_zf} holds for $f$.
\end{lem}

\begin{proof}
Take a number field $K$ where all concerned are defined.
It follows from Lemma \ref{lem3.8} (ii) that  
$Z_f(K) \subset \pi^{-1}(Z_g(K))$.
Applying Theorem \ref{thm4.1} to $g$, it follows that 
$Z_g(K)=\Preper_g(K)$ 
and $S=\Preper_g(K)$ is finite.
So $Z_f(K)$ 
is contained in the $f$-invariant proper closed subset 
$\pi^{-1}(S)$.
\end{proof}

\begin{proof}[Proof of Theorem \ref{thm7.7}]
Assume that $X$ has a $(-1)$-curve $E$.
By Lemma \ref{lem7.3}, $f^N(E)=E$ for some $N \in \mathbb Z_{>0}$.
Using Lemma \ref{lem3.5}, we may assume that $f(E)=E$ by replacing $f$ by $f^N$.
Let $\mu: X \to Y$ be the contraction of $E$.
Then an endomorphism $g$ on $Y$ satisfying $\mu \circ f=g \circ \mu$ is induced.
Lemma \ref{lem7.2} implies that 
$Z_{f}(K) \subset \pi^{-1}(Z_g(K))$ for any sufficiently large number field $K$.
Assume that $Z_g(K) \subset W(K)$
for a $g$-invariant proper closed subset $W \subset Y$.
Then $V=\pi^{-1}(W)$ is an $f$-invariant proper closed subset of $X$ satisfying  
$Z_f(K) \subset V(K)$. 
This argument shows that the proof of the theorem for $f$ is reduced to that for $g$.

Continuing this reduction process, we may assume that $X$ has no $(-1)$-curve.
By Lemma \ref{lem3.5} and Lemma \ref{lem3.8.1},
it is sufficient  to apply operations (a) and (b) in Theorem \ref{thm7.4} to $(X,f)$ 
and prove the assertion for the replaced ones.
\begin{itemize}
\item
If $X=\mathbb P^2$,
then $l_f=0$ and Conjecture \ref{conj_zf} holds for $f$ by Theorem \ref{thm4.1}.
\item
If $X$ is a $\mathbb P^1$-bundle, then $\rho(X)=2$ and so 
$l_f=0$ by Theorem \ref{thm4.2} (i).
\begin{itemize}
\item
If $X$ is a $\mathbb P^1$-bundle over a curve of genus $\geq 2$,
then $X$ is not potentially dense, 
and so Conjecture \ref{conj_zf} trivially holds.
\item
If $X$ is a Hirzebruch surface or a $\mathbb P^1$-bundle over an elliptic curve,
then $(X,f)$ is reduced into type (II) by Theorem \ref{thm7.4}.
So there is an endomorphism $g$ on a curve $C$ with $\delta_g=\delta_f$ 
and a surjective morphism $\pi: X \to C$ such that $\pi \circ f=g \circ \pi$.
Since $l_f=0$ and $\delta_f=\delta_g$,
it follows from Lemma \ref{lem3.8} (i) that $l_g \leq l_f=0$.
Then Lemma \ref{lem7.6} implies that Conjecture \ref{conj_zf} holds for $f$.
\end{itemize}
\item
If $X$ is an abelian surface,
the claim is a special case of Theorem \ref{thm5.1}.
\item
If $X$ is a bielliptic surface,
$\rho(X) \leq h^{1,1}(X)=2$ and so $l_f=0$ by Theorem \ref{thm4.2} (i).
By Theorem \ref{thm7.4}, $(X,f)$ is reduced into type (II).
So there is an endomorphism $g$ on a curve $C$ with $\delta_g=\delta_f$ 
and a surjective morphism $\pi: X \to C$ such that $\pi \circ f=g \circ \pi$.
Since $l_f=0$ and $\delta_f=\delta_g$,
it follows from Lemma \ref{lem3.8} (i) that $l_g \leq l_f=0$.
Then Lemma \ref{lem7.6} implies that Conjecture \ref{conj_zf} holds for $f$.
\item
If $X$ is a properly elliptic surface,
then it follows from \cite[Theorem 3.2]{Fuj02} that 
there is an elliptic curve $E$ and a curve $C$ of genus 
$\geq 2$ such that 
$E \times C$ is an \'etale cover of $X$.
Here $E \times C$ is not potentially dense 
(if $E \times C$ is potentially dense, then $C$ is also potentially dense,
but this contradicts Faltings's theorem).
So $X$ is not potentially dense due to Theorem \ref{thm_CW}.
Hence Conjecture \ref{conj_zf} trivially holds for $f$.
Moreover, Theorem \ref{thm7.4} and Lemma \ref{lem7.5} show that $l_f=0$.
\end{itemize}

Eventually, $l_f=0$ if $X$ is not birational to an abelian variety,
and Conjecture \ref{conj_zf} holds for $f$ in any case.
\end{proof}

\section{Applications}\label{sec8}

In this section, we obtain two applications of ample canonical heights.

As we saw in the introduction, the Call--Silverman canonical height 
for a polarized endomorphism is used to show that the number of preperiodic points 
over any fixed number field is finite.
For general endomorphisms, our main conjecture (Conjecture \ref{conj_zf}) implies 
the non-density of preperiodic points over any fixed number fields:

\begin{prop}\label{prop8.0}
Let $X$ be a smooth projective variety and $f$ an endomorphism on $X$ with $\delta_f>1$.
Assume that Conjecture \ref{conj_zf} holds for $f$.
Then $\Preper_f(K)$ is not Zariski dense for any number field $K$.
\end{prop}

\begin{proof}
It is clear that $\Preper_f(K) \subset Z_f(K)$ for any subfield 
$K \subset \overline{\mathbb Q}$.
So the assertion follows.
\end{proof}

Therefore Theorem \ref{thm_main} deduces the following.

\begin{thm}\label{thm8.0.1}
Let $X$ be a smooth projective variety and $f$ an endomorphism on $X$ with $\delta_f>1$.
Assume that $(X,f)$ satisfies one of the following conditions.
\begin{itemize}
\item  
$f^*H \equiv \delta_f H$ for an ample $\mathbb R$-divisor $H$ on $X$.
\item 
$\rho(X) \leq 2$ and $f$ is an automorphism.
\item 
$X$ is an abelian variety.
\item 
$X$ is a smooth projective surface.
\end{itemize}
Then $\Preper_f(K)$ is not Zariski dense for any number field $K$.
\end{thm}

\begin{rem}
(i) As we saw in Section \ref{sec4},
$\Preper_f(K)$ is finite for any number field $K$ in the first two cases.

(ii) We can also prove the abelian variety case by using the nef canonical height 
(cf.~\cite[Theorem 1]{KaSi16b}).
\end{rem}

Let us see another application of ample canonical heights.
Using ample canonical heights, we can investigate the intersection $O_f(x) \cap O_g(y)$ of  
two dense orbits $O_f(x), O_g(y)$ of two endomorphisms on a variety.
The results and arguments in this section are based on the argument appearing in 
\cite[Theorem 5.11.0.1]{BGT16}.

\begin{thm}\label{thm8.1}
Let $X$ be a smooth projective variety and $f,g$ endomorphisms on $X$ 
such that $\delta_f=\delta_g>1$ and $l_f=l_g$.
Assume that Conjecture \ref{conj_main} holds for $f$ and $g$.
Take a dense $f$-orbit $O_f(x)$ and  a dense $g$-orbit $O_g(y)$.
Then the set 
$\{ |n-m| \mid n,m \in \mathbb Z_{\geq 0},\ f^n(x)=g^m(y) \}$ is upper bounded.
\end{thm}

\begin{rem}
The proof of Theorem \ref{thm8.1} is similar to the proof of 
\cite[Theorem 5.11.0.1]{BGT16},
where  polarized endomorphisms are treated.
\end{rem}

\begin{proof}[Proof of Theorem \ref{thm8.1}]
Set $(\delta,l)=(\delta_f,l_f)$.
Since Conjecture \ref{conj_main} holds for $f$, we have $\underline h_f(x)>0$.
So there is $\varepsilon >0$ such that 
$\delta^{-n} n^{-l} h_X(f^n(x)) \geq \varepsilon$ for every $n \in \mathbb Z_{\geq 0}$.
Moreover, since $\overline h_g(y)<\infty$,
there is $C>0$ such that 
$\delta^{-n} n^{-l} h_X(g^n(x)) \leq C$ for every $n \in \mathbb Z_{\geq 0}$.
Take $n,m \in \mathbb Z_{\geq 0}$ such that $n \geq m$ and $f^n(x)=g^m(y)$.
Then we have 
$$\delta^{n-m} \varepsilon \leq \delta^{n-m} \frac{h_X(f^n(x))}{\delta^n n^l}
= \frac{h_X(g^m(y))}{\delta^m m^l} \left( \frac{m}{n} \right)^l \leq C.$$
So $n-m \leq \log_{\delta}(C/\varepsilon)$.
Similarly, for $n,m \in \mathbb Z_{\geq 0}$ such that $n \leq m$ and $f^n(x)=g^m(y)$,
$m-n$ is upper bounded.
Hence the claim follows.
\end{proof}

We need the following dynamical Mordell--Lang theorem for \'etale endomorphisms 
due to Bell--Ghioca--Tucker.

\begin{thm}[The dynamical Mordell--Lang theorem for \'etale maps, 
{\cite[Theorem 1.3]{BGT10}}]\label{thm_BGT}
Let $X$ be a projective variety, $f$ an \'etale endomorphism on $X$,
and $V$ a closed subvariety of $X$.
Then the set $\{ n \in \mathbb Z_{\geq 0} \mid f^n(x) \in V \}$ is a finite union of sets 
of the form $\{ kn+i\}_{n=0}^\infty$ 
for some $k,i \in \mathbb Z_{\geq 0}$.
\end{thm}

Using this theorem,
we can obtain a sharper description of the intersection $O_f(x) \cap O_g(y)$ 
if we assume that $f,g$ are \'etale.

\begin{thm}\label{thm8.2}
Let $X$ be a smooth projective variety and $f,g$ \'etale endomorphisms on $X$ 
such that $\delta_f=\delta_g>1$ and $l_f=l_g$.
Assume that Conjecture \ref{conj_main} holds for $f$ and $g$.
Take a dense $f$-orbit $O_f(x)$ and  a dense $g$-orbit $O_g(y)$.
Then the set $\{ (n,m) \in (\mathbb Z_{\geq 0})^2 \mid f^n(x)=g^m(y) \}$ 
is a finite union of sets of the form $\{ (kn+i,kn+j) \}_{n=0}^\infty$ 
for some $k,i,j \in \mathbb Z_{\geq 0}$.
\end{thm}

\begin{rem}\label{rem8.2.1}
Theorem \ref{thm8.2} essentially says that the intersection of 
two orbits with same height growth has a nice form.
Sano \cite[Theorem 1.2]{San18} proved that 
the intersection of two orbits has a nice form under a weaker assumption on 
height growth of the orbits.
\end{rem}

\begin{proof}[Proof of Theorem \ref{thm8.2}]
Theorem \ref{thm8.1} implies that 
$N=\max \{ |n-m| \mid n,m \in \mathbb Z_{\geq 0},\ f^n(x)=g^m(y) \} < \infty$.
Fix $l \in \{0,1,\ldots,N\}$.
For $n \in \mathbb Z_{\geq 0}$,
$f^{n+l}(x)=g^n(y)$ if and only if $(f \times g)^n((f^l(x),y)) \in \Delta$, where 
$\Delta \subset X \times X$ is the diagonal set.
Moreover, $f \times g$ is an \'etale endomorphism on $X \times X$.
So Theorem \ref{thm_BGT} implies that 
$\{ (n+l,n) \mid n \in \mathbb Z_{\geq 0},\ f^{n+l}(x)=g^n(y) \}$ is a finite union of sets 
of the form 
$\{ (kn+i+l,kn+i)\}_{n=0}^\infty$ for some $k,i \in \mathbb Z_{\geq 0}$.
Similarly,
$\{ (n,n+l) \mid n \in \mathbb Z_{\geq 0},\ f^{n}(x)=g^{n+l}(y) \}$ is a finite union of sets 
of the form 
$\{ (kn+i,kn+i+l)\}_{n=0}^\infty$ for some $k,i \in \mathbb Z_{\geq 0}$.
Therefore 
\begin{align*}
&\{ (n,m) \in (\mathbb Z_{\geq 0})^2 \mid f^n(x)=g^m(y) \} \\
= &\bigcup_{l=0}^N \{ (n+l,n) \mid n \in \mathbb Z_{\geq 0},\ f^{n+l}(x)=g^n(y) \} \\
&\cup \bigcup_{l=0}^N \{ (n,n+l) \mid n \in \mathbb Z_{\geq 0},\ f^{n}(x)=g^{n+l}(y) \}
\end{align*}
is a finite union of sets of the form  $\{ (kn+i,kn+j)\}_{n=0}^\infty$ 
for some $k,i,j \in \mathbb Z_{\geq 0}$.
\end{proof}

Applying Theorem \ref{thm8.1} and Theorem \ref{thm8.2} 
to the endomorphisms on the varieties which we have considered, 
we obtain the following as an application of Theorem \ref{thm_main}.

\begin{thm}
Let $X$ be a smooth projective variety and $f,g$ endomorphisms on $X$ such that $\delta_f=\delta_g>1$ and $l_f=l_g$.
We assume one of the following: 
\begin{itemize}
\item 
$f^*H \equiv \delta_f H$ and  $g^*H' \equiv \delta_g H'$ 
for some ample $\mathbb R$-divisors $H, H'$ on $X$,
\item 
$\rho(X) \leq 2$ and $f,g$ are automorphisms,
\item 
$X$ is an abelian variety, or
\item 
$X$ is a smooth projective surface.
\end{itemize}
Take a dense $f$-orbit $O_f(x)$ and  a dense $g$-orbit $O_g(y)$.
Then the set $\{ |n-m| \mid n,m \in \mathbb Z_{\geq 0},\ f^n(x)=g^m(y) \}$ is upper bounded.
Furthermore, if both $f$ and $g$ are \'etale, 
then the set $\{ (n,m) \in (\mathbb Z_{\geq 0})^2 \mid f^n(x)=g^m(y) \}$ 
is a finite union of sets of the form $\{ (kn+i,kn+j)\}_{n=0}^\infty$ 
for some $k,i,j \in \mathbb Z_{\geq 0}$.
\end{thm}

%


\end{document}